\documentclass[12pt]{article}

\usepackage{amssymb,amsfonts,amsmath,amsthm}
\usepackage{color}
\usepackage[colorlinks = true,
linkcolor = red,
urlcolor  = red,
citecolor = blue,
anchorcolor = red]{hyperref}

\numberwithin{equation}{section}

\newtheorem{Assumption}{Assumption}[section]
\newtheorem{theorem}{Theorem}[section]
\newtheorem{lemma}[theorem]{Lemma}
\newtheorem{proposition}[theorem]{Proposition}
\newtheorem{Corollary}[theorem]{Corollary}
\newtheorem{remark}[theorem]{Remark}
\newtheorem{definition}{Definition}[section]

\begin{document}
	
	\title{\bf Dynamics for a viscoelastic
		beam equation with past history and nonlocal boundary dissipation}
	
	\author{
		\small {\bf Linfang Liu}  \footnote{Email: liulinfang2020@nwu.edu.cn}
		\footnote{The work of Lin F. Liu has been partially supported by NSF of China (Nos. 11901448) as well as by Scientific Research  Foundation of Northwest University of China. }\\
		\small School of Mathematics, Northwest University, \\
		\small Xi'an 710119, Shaanxi, People's Republic of China.\medskip\\
		\small {\bf Vando Narciso}  \footnote{Corresponding author. Email: vnarciso@uems.br.}
		\footnote{Supported by the Fundect/CNPq, Grant 15/2024.}\\
		\small Center of Exact and Earth Sciences, State University of Mato Grosso do Sul,\\
		\small 79804-970, Dourados, MS, Brazil.\medskip\\
		\small{{\bf Zhijian Yang}\thanks{ Supported by the National
				Natural Science Foundation   of China (Grant No.12171438). E-mail: yzjzzvt@zzu.edu.cn}}\\
		{\small School of Mathematics and Statistics, Zhengzhou University,} \\
		{\small Zhengzhou 450001, China.}}
	
	\date{}
	\maketitle

	\begin{abstract}
		This article aims to study the long-time dynamics of the linear viscoelastic plate equation
		$\displaystyle{u_{tt}+\Delta^2 u-\int_{\tau}^tg(t-s)\Delta^2u(s)ds=0}$ subject to nonlinear and nonlocal boundary conditions. This model, with $\tau=0$, was first considered by Cavalcanti (Discrete Contin. Dyn. Syst., 8(3), 675-695, 2002), where results of global existence and uniform decay rates of energy have been established. In this work, by taking $\tau=-\infty$, and considering the autonomous equivalent problem we prove that the dynamical system $(\mathcal{H},S_t)$ generated by the weak solutions has a compact global attractor $\mathfrak{A}$ (in the topology of the weak phase space $\mathcal{H}$), which in subcritical case has finite dimension and smoothness. Furthermore, when the force follows the {\it Hook Law}, we prove that $(\mathcal{H},S_t)$ possesses a (generalized) fractal exponential attractor $\mathfrak{A}_{\exp}$ with finite dimension in a space $\widetilde{\mathcal{H}}\supset\mathcal{H}$.
	\end{abstract}

	
	\smallskip
	
	\noindent{{\bf 2010 MSC:}  35B40, 35B41, 35L75.
		
		\section{Introduction}
		\subsection{The model and its physical interpretation}
		From a physical point of view, it is known that	viscoelastic materials produce natural dissipative effects, which is due to the special property of these materials to maintain the memory of their past history, and, the mathematical expressions of these damping effects are described by integro-differential operators. In one-dimensional space, this context of viscoelastic effects applied to the study of the vibration of a beam can be represented by the following model
		$$\rho Au_{tt}+\alpha u_{xxxx}-\int_0^tg(t-s)u_{xxxx}(x,s)ds=F(x, t, u, u_t),$$
		where $\rho, A$, and $\alpha$ denote, respectively, the density, the cross-sectional area, and the tension stiffness. The function $F$ represents the action of source and dissipative terms and $g$ is so-called relaxation function. Motivated by this class of problems, we propose in this work to study the asymptotic behavior of the following class of beam/plate equation with viscoelastic effects and subject to nonlinear boundary conditions
		\begin{eqnarray*}\left\{\begin{array}{l}\label{P}\displaystyle{u_{tt}+\Delta^2 u-\int_{\tau}^tg(t-s)\Delta^2u(s)ds=0,\quad \mbox{in}\quad\Omega\times[0,\infty),}\\
				\displaystyle{u=\frac{\partial u}{\partial \nu}=0\quad \mbox{on}\quad \Gamma_0\times[0,\infty),}\\
				\displaystyle{\Delta u-\int_{\tau}^tg(t-s)\Delta u(s)ds=0\quad\mbox{on}\quad \Gamma_1\times[0,\infty),}\\
				\displaystyle{\frac{\partial(\Delta u)}{\partial \nu}-\int_{\tau}^tg(t-s)\frac{\partial(\Delta u(s))}{\partial\nu}ds=f(u)+a(t)u_t\quad\mbox{on}\quad \Gamma_1\times[0,\infty),}\\
				u(x,0)=u_{0}(x),\quad u_t(x,0)=u_1(x),\quad x\in  \Omega,
			\end{array}\right.\quad(*)\end{eqnarray*}
		where $\Omega\subset \mathbb{R}^n$, $n\geq1$, is a nonempty bounded open set  with smooth boundary $\partial\Omega=\Gamma=\Gamma_0\cup\Gamma_1$. Here, $\Gamma_0,\Gamma_1$ are closed and $\Gamma_0\cap\Gamma_1=\emptyset$, with $meas(\Gamma_0)\neq0$, and $\nu$ is the unit outward normal to $\Gamma$. $a(t)$  is a nonlocal coefficient given by $a(t)=M\left(\int_{\Gamma_1}|u(x,t)|^2d\Gamma\right),$ where $M\in C^1(\mathbb{R}^+)$ is a non-degenerate function, $f(u)$ is a nonlinear sourcing term like $f(u)\sim |u|^{p}u-C$, and $g$ represents the kernel of the memory term which will be assumed to decay exponentially. This problem was first studied by Cavalcante in \cite{Cavalcanti-2002}, with $\tau=0$ and assuming  $f(u)=k|u|^{\gamma}u$. Results of global existence and uniform energy decay were obtained. In the physical sense, Eq. $(*)$ is interpreted by the deflection $u(x,t)$ of a viscoelastic beam (when $n=1$) or a viscoelastic plate (when $n=2$), which in relation to its boundary is clamped in the internal portion $\Gamma_0$ and on the external portion $\Gamma_1$ is supported by a bearing with nonlinear responses characterized by the function $f(u)$ and is also subject to frictional dissipation with intensity dependent on the nonlocal coefficient $a(u,t)$, which assumes the distributional average on the whole portion $\Gamma_1$ of the boundary $\Gamma$.

		\subsection{Equivalent autonomous system}
		Our main objective of the present work is to extend the analysis of problem $(*)$ to the context of the dynamics of $(*)$ associated with the case in which the set of stationary points of the problem is non-trivial. In other words, the analysis of $(*)$  in the sense of the existence of global attractors. However, system $(*)$ is non-autonomous. To get around this situation, let's assume problem $(*)$  with $\tau=-\infty$. Then, following Dafermos \cite{Dafermos}, we shall add a new variable $\eta^t$ to the systems which corresponds to the relative displacement history
		$$\eta^t=\eta^t(x,s)=u(x,t)-u(x,t-s),\quad (x,s)\in \Omega\times \mathbb{R}^+,\;t\ge 0.$$
		Differentiating the above equation with respect to $t$, we get
		$$\eta^t_t(x,s)=-\eta^t_s(x,s)+u_t(x,t),\quad (x,s)\in \Omega\times \mathbb{R}^+,\;t\ge 0,$$
		and we can take as initial condition $(t=0)$
		$$\eta^0(x,s)=u_0(x,0)-u_0(x,-s),\quad (x,s)\in \Omega\times \mathbb{R}^+.$$
		By noting that $\Delta^2\eta^t(s)=\Delta^2 u(x,t)-\Delta^2 u(x,t-s)$, the original memory term can be rewritten as
		\begin{eqnarray*}
			-	\int_{-\infty}^tg(t-s)\Delta^2u(s)ds&=&-\int_0^{\infty}g(\tau)\Delta^2u(t-\tau)d\tau\\
			&=&\int_0^{\infty}g(\tau)\Delta^2\eta^t(\tau)d\tau-\int_{0}^{\infty}g(\tau)d\tau\Delta^2u(t).
		\end{eqnarray*}
		Hence, taking for simplicity $\kappa:=1-\int_0^{\infty}g(\tau)d\tau$, the non-autonomous system \eqref{P} can be converted in the following autonomous system
		\begin{eqnarray}\left\{\begin{array}{l}\label{P}\displaystyle{u_{tt}+\kappa\Delta^2 u+\int_{0}^{\infty}g(\tau)\Delta^2\eta^t(\tau)d\tau=0,\quad \mbox{in}\quad \Omega\times \mathbb{R}^+,}\\
				\displaystyle{\eta^t_t=-\eta^t_s+u_t,\quad \mbox{in}\quad\Omega\times \mathbb{R}^+\times\mathbb{R}^+,}
			\end{array}\right.\end{eqnarray}
		with  boundary condition
		\begin{eqnarray}
			\left\{\begin{array}{l}
				\label{boundary-condition}
				u=\frac{\partial u}{\partial \nu}=0\quad \mbox{on}\quad \Gamma_0\times\mathbb{R}^+\quad\mbox{and}\quad \eta^t=\frac{\partial \eta^t}{\partial \nu}=0\quad \mbox{on}\quad \Gamma_0\times\mathbb{R}^+\times\mathbb{R}^+,\\
				\displaystyle{\kappa\Delta u+\int_{0}^{\infty}g(\tau)\Delta \eta^t(\tau)d\tau=0\quad\mbox{on}\quad \Gamma_1\times\mathbb{R}^+\times\mathbb{R}^+,}\\
				\displaystyle{\kappa\frac{\partial(\Delta u)}{\partial \nu}+\int_{0}^{\infty}g(\tau)\frac{\partial(\Delta \eta^t(\tau))}{\partial\nu}d\tau=f(u)+a(t)u_t\quad\mbox{on}\quad \Gamma_1\times\mathbb{R}^+\times\mathbb{R}^+,}
			\end{array}\right.
		\end{eqnarray}
		and initial conditions
		\begin{eqnarray}
			\left\{\begin{array}{l}
				\label{initial-condition0}
				u(x,0)=u_{0}(x),\quad u_t(x,0)=u_1(x),\quad x\in  \Omega,\\
				\eta^0(x,s)=\eta_0(x,s),\quad\mbox{in}\quad \Omega\times \mathbb{R}^+,
			\end{array}\right.
		\end{eqnarray}
		where the history is considered as an initial value
		$$\eta_0(x,s)=u_0(x,0)-u_0(x,-s),\ \text{in}\ \ \Omega\times\mathbb{R}^+.$$
		
		Thus, the analysis of the existence of a global attractor is well defined for the autonomous system \eqref{P}-\eqref{initial-condition0}. More specifically, the results established in this work are:
		\begin{itemize}
			\item[(a)] We establish the results of existence and uniqueness of global solutions for problem \eqref{P}-\eqref{initial-condition0} (Theorem \ref{theo-existence}) which ensures that the weak solutions define a dynamical system $(\mathcal{H},S_t)$ in the weak topology of \eqref{P}-\eqref{initial-condition0}.
			\item[(b)] We prove that the generated dynamical system $(\mathcal{H},S_t)$ is gradient (Proposition \ref{Prop-grad-syst}) and dissipative (Proposition \ref{absorbing-set}).
			\item[(c)] In the subcritical case $p< p^*$, we prove that the system is quasi-stable (Corollary \ref{Cor-quasi-stable}), which guarantees the existence of a compact global attractor $\mathfrak{A}$ that coincides with the unstable manifold $\mathcal{M}(\mathcal{N})$ emanating from the set of stationary points $\mathcal{N}$ of the dynamical system $(\mathcal{H},S_t)$. Moreover, the quasi-stability property also ensures that the attractor $\mathfrak{A}$ has properties such as finite dimension and smoothness.
			\item[(e)] Finally, for $f(u)\sim\lambda u$, we prove that there exists a space $\widetilde{\mathcal{H}}\supset\mathcal{H}$ such that $t\mapsto S_ty$ is H\"older continuous in $\widetilde{\mathcal{H}}$ for every $y\in \mathcal{B}$ ($\mathcal{B}$ absorbing set) (Proposition \ref{Prop-exp-attractor}) which together with the quasi-stability property guarantees the existence of a fractal exponential attractor $\mathfrak{A}_{\exp}$ whose dimension is finite in the space $\widetilde{\mathcal{H}}$.
		\end{itemize}
		
		The work is justified, as this is the first analysis of the asymptotic dynamics of the problem \eqref{P}-\eqref{initial-condition0} in the sense of the existence of global attractors and its properties. 

  \medskip
  It is also important to highlight that we improved, in relation to \cite{Cavalcanti-2002}, the exponent of the force term $f$. In \cite{Cavalcanti-2002} the results are obtained for $p$ such that
	$$0<p\le \frac{1}{n-2}\quad \mbox{if}\quad n\ge 3\quad \mbox{and}\quad p>0\quad \mbox{if}\quad n=1,2.$$
 In this work, the well-posedness and existence of a dissipative gradient system are obtained for $p$ such that
 $$0<p\le p^*:=\frac{3}{n-4}\quad \mbox{if}\quad n\ge 5\quad \mbox{and}\quad p>0\quad \mbox{if}\quad 1\le n\le 4.$$
 And, existence of attractor and its properties for the sub-critical case $p<p*$.
 
		\medskip
		It is well known that second-order evolution equations with viscoelastic effects have received the attention of several researchers in recent years and it is very difficult to quantify the number of papers in this regard. Therefore, in what we will present the pioneering work in relation to beam/plate equations with nonlinear boundary and the pioneering study in relation to nonlocal dissipation.
		
		\medskip
		Regarding the dissipative terms of the system \eqref{P}-\eqref{initial-condition0}, it is important to highlight that: although the problem has two dissipative terms, an internal one given by the memory term, and an external one, given by the nonlocal term at the boundary, the dissipation that governs the attraction rate is the internal one. On the other hand, the dissipation at the boundary does not have the property of monotonicity, that is, the nonlocal term is divided into a monotonic (good) part and a non-monotonic part that generates a (bad) term that needs control. And this control is well established by the good part of the nonlocal damping itself.
		
		\subsection{Pioneering work}
		Pioneering work in the study of modeling problems of vibrating beams in elastic bearings were the papers of Feireisl \cite{Feireisl} and Feckan \cite{Feckan}, where existence results of  time-periodic solutions of the following problem were studied
		\begin{eqnarray}\label{Problem0}
			\left\{\begin{array}{l}
				u_{tt}+u_{xxxx}=0,\quad \mbox{in}\quad [0,L]\times \mathbb{R}^+,\\
				u_{xx}(0,t)=u_{xx}(L,t)=0,\\
				u_{xxx}(0,t)=-f(u(0,t)),\\
				u_{xxx}(L,t)=f(u(L,t)).
			\end{array}\right.
		\end{eqnarray}
		Soon after, analogue problems were considered by other authors. For example, Grossinho and Ma \cite{Grossinho-ma} established results of symmetric solutions by using critical point theory and Fenchel-Legendre transform. Grossinho and Tersian \cite{Grossinho-tersian} studied a similar problem assuming a discontinuous elastic foundation. Ma \cite{Matofu-2000} studies the stationary problem related to problem \eqref{Problem0}, where necessary and sufficient conditions for the existence of solutions are established when $f$ and $g$ are monotone nonlinearities. In \cite{Matofu-2001,Matofu-narciso-pelicer} problem \eqref{Problem0} was considered under the presence of a nonlocal Kirchhoff term $M(\int_0^L|u_x|^2dx)u_{xx}$ and with monotone nonlinearities $f$ and $g$. Stability results and existence of global attractor were established. Here it is also important to cite the work by Ji and Lasiecka \cite{Ji-lasiecka} who studied the following plate equation with rotational forces
		\begin{eqnarray}
			\left\{\begin{array}{l}\label{problem1}
				u_{tt}-\gamma \Delta u_{tt}+\Delta^2u+f(u)=0\quad \mbox{in}\quad \Omega\times \mathbb{R}^+,\\
				u=0,\quad\Delta u=-g\left(\frac{\partial u_t}{\partial\nu}\right)\quad \mbox{in}\quad \partial \Omega\times \mathbb{R}^+,\\
				u(x,0)=u^0(x),\quad u_t(x,0)=u_1(x),\quad \mbox{in}\quad \Omega.
			\end{array}\right.
		\end{eqnarray}
		The authors show that \eqref{problem1} is uniformly stabilizable with uniform energy decay rates with respect to the parameter $\gamma>0$ sufficiently small. Furthermore, they prove that the solutions of the problem with rotational forces converge to the solutions of the Euler-Bernoulli plate, when $\gamma\rightarrow 0$. Here, it is also important to highlight the work regarding nonlinear boundary conditions for wave equations \cite{lasiecka-tataru,Irena-Ong}.
		\medskip
		On the other hand, for pioneering works concerning viscoelastic beam/plate equations, we refer the following works: Leugering \cite{Leugering} proved, through harmonic analysis arguments, the exact controllability for a square viscoelastic plate with long memory. Lagnese \cite{Lagnese-1991} studied the viscoelastic plate equation subject to non-linear feedback and showed that the solution of the corresponding equation decays to zero as time goes to infinity. Rivera {\it et al.} \cite{Rivera} who studied the Kirchhoff linear viscoelastic plate and proved that the first and second associated energies decay exponentially (or polynomially) when the memory core also decays exponentially (or polynomially).
		
		\medskip
		Now, the presence of a nonlocal dissipation of the Kirchhoff type appears in the model
		\begin{align}\label{model-lange-menzala}u_{tt} + \Delta^2 u + a(t)u_t = 0 \quad\mbox{in}\quad \mathbb{R}^n\times \mathbb{R}^+,
		\end{align}
		proposed by Lange and Menzala \cite{Lange-menzala} as a type of nonlocal beam equations. By using Fourier transform and exploring the regularity of initial data, they proved the existence of global classical solutions and algebraic energy decay. Cavalcanti {\it et al.} \cite{Cavalcanti-2001} studied equation \eqref{model-lange-menzala} in a viscoelastic formulation by adding a memory term of fourth order
		\begin{align*}
			u_{tt}+\Delta^2 u-\int_0^tg(t-s)\Delta^2u(s)ds+a(t)u_t=0
			\quad\mbox{in}\quad \Omega\times \mathbb{R}^+,\end{align*}
		where $\Omega\subset \mathbb{R}^n$ is a bounded or unbounded open set.
		They showed global existence and energy decay. Later, Cavalcanti in \cite{Cavalcanti-2002} makes the first study of this nonlocal class of dissipation in the proposed model (*).

		\section{Well-posedness}
		\subsection{Notations and assumptions}

		\medskip
		Now we will define some notations that will be used throughout the work. For the standard $L^p(\Omega)$ space, we denote
		$$\|u\|^p_p=\int_{\Omega}|u(x)|^pdx,\quad \|u\|^p_{p,\Gamma}=\int_{\Gamma}|u(x)|^pd\Gamma,\ p\geq1,$$ and in the particular case $p=2,$
		$$\left(u,v\right)=\int_{\Omega}u(x)v(x)dx,\quad \|u\|^2=\int_{\Omega}|u(x)|^2dx,\quad \mbox{and}\quad \|u\|^2_{\Gamma}=\int_{\Gamma}|u(x)|^2d\Gamma.$$
		We denote by $L^2_{g}(\mathbb{R}^+;X)$ the Hilbert space of $X$-valued functions on $\mathbb{R}^+$, endowed with the inner product and norm, respectively,
		$$(\varphi,\psi)_{g,X}=\int_0^{\infty}g(\tau)\left(\varphi(\tau),\psi(\tau)\right)_Xd\tau\quad \mbox{and}\quad \|\varphi\|^2_{L^2_{g}(\mathbb{R}^+;X)}=\int_0^{\infty}g(\tau)\|\varphi(\tau)\|_X^2d\tau.$$
		In the particular case $X=L^2(\Omega)$, we denoted by
		$$(\varphi,\psi)_{g,L^2(\Omega)}:=(\varphi,\psi)_{g}=\int_0^{\infty}g(\tau)\left(\varphi(\tau),\psi(\tau)\right)d\tau.$$
		With respect to the Sobolev spaces which will be used, let
		\begin{eqnarray*}V&:=&\{v\in H^2(\Omega);v=\frac{\partial v}{\partial \nu}=0\;\;\mbox{on}\;\; \Gamma_0\},\\
			H&:=&\{v\in H^4(\Omega)\cap V;\,\Delta v=0\;\mbox{on}\; \Gamma_1\},\end{eqnarray*}
		equipped with inner products,
		\begin{eqnarray*}(u,v)_V=\left(\Delta u,\Delta v\right),\quad (u,v)_H=(u,v)_V+\left(\Delta^2 u,\Delta^2 v\right),\end{eqnarray*}
		respectively.
		
		Motivated by the boundary condition \eqref{boundary-condition} we assume, for regular solutions,  that
		\begin{align*}
			z_0=(u_0,u_1,\eta_0)\in H\times H\times L^2_g(\mathbb{R}^+;H)
		\end{align*}
		verifying the compatibility conditions
		\begin{align}\label{comp-cond}
			\frac{\partial{(\Delta u_{0})}}{\partial\nu}+\int_0^{\infty}g(\tau)\frac{\partial( \Delta \eta_{0}(\tau))}{\partial\nu}d\tau =f(u_{0})+M(\|u_{0}\|^2_{\Gamma_{1}})u_{1}\ \ \text{on}\ \ \Gamma_1.
		\end{align}
		Let
		\begin{align*}
			\mathcal{H}_1:=\left\{z_0=(u_0, u_{1},\eta_0)\in H\times H\times L^2_g(\mathbb{R}^+;H);\ z_0\;\;\mbox{satisfies condition}\;\; \eqref{comp-cond}
			\right\}.
		\end{align*}
		Then, the natural choice of the phase space is the following one:
		$$\mathcal{H}:= \overline{\mathcal{H}_1}^{V\times L^{2}(\Omega)\times L^2_{g}(\mathbb{R}^+;V)},$$
		equipped with the adapted norm
		$$||z||^2_{\mathcal{H}}\equiv ||z||^2_{V\times L^2(\Omega)\times L^{2}_g(\mathbb{R}^+;V)}=\kappa\|\Delta u\|^2+\|v\|^2+\|\eta^t\|^2_{L^2_{g}(\mathbb{R}^+;V)},\quad z=(u,v,\eta^t).$$

		We state the general hypothesis in the following.
		\begin{Assumption}\label{A1}\rm
			\begin{itemize}
				\item
				For nonlocal damping term
				$
				a(t):=M(\|u(t)\|^2_{\Gamma_1}),
				$
				we assume that $M\in C^1(\mathbb{R}^+)$ is a non-negative function such that
				\begin{align}\label{hyp_M}M(s)\ge m_0>0,\quad \forall s\in \mathbb{R}^+.\end{align}
			\item With respect to the nonlinear term $f$, we assume that $f\in C^{1}(\mathbb{R})$, and
			\begin{eqnarray}
				&&|f'(u)|\leq C_{f'}(1+|u|^p),\ u\in\mathbb{R},\label{diff}\\
				&&-C_{f}-\frac{c_f}{2}|u|^2\leq F(u)=\int_{0}^{u}f(s)ds\leq f(u)u+\frac{c_f}{2}|u|^2, \ u\in\mathbb{R},\label{diff2}
			\end{eqnarray}
			for some positive constants $C_{f'}$, $C_{f}$, $c_f\in[0,\min\{\frac{\kappa}{\lambda_0},\frac{\kappa}{\lambda_1}\})$ and growth exponent $p$ such that
			\begin{align*}
				 0<p<\frac{3}{n-4}\ \ \text{if}\ n\geq 5\ \ \text{and}\ \ p>0\ \text{if}\ 1\leq n\leq 4.
			\end{align*}
			\item For the kernel $g$, we assume that $g:\mathbb{R}_{+}\rightarrow\mathbb{R}_{+}$ is a bounded $C^2$ function
			and there exist positive constants $\alpha_1$, $\alpha_2$ and $\alpha_3$ satisfying
			\begin{align}\label{kernel_g'}
				-\alpha_1 g(t)&\leq g'(t)\leq-\alpha_2 g(t),\\
				0&\leq g''(t)\leq\alpha_3 g(t).
			\end{align}
			\item $\varphi(x,s):=u_0(x,-s)$ is such that
			\begin{align}\label{assumption-varphi}
				\varphi,\varphi'\in L^2_g(\mathbb{R}^+;V).
			\end{align}
		\end{itemize}
	\end{Assumption}
	\begin{remark}
		It is easy to check that $g$ is decreasing, with $\lim\limits_{t\rightarrow\infty}g(t)=0$ and $g(0)>0$.
	\end{remark}

		\begin{remark}
			Let $u\in H^2(\Omega)$. By the Sobolev  embeddings, (see Lemma \ref{5.5})
			\begin{eqnarray}\label{imbeddings}\left\{\begin{array}{rcl} H^2(\Omega)\hookrightarrow L^{\frac{2n}{n-4}}(\Omega),\;\; &H^2(\Omega)\hookrightarrow H^{3/2}(\Gamma) \hookrightarrow L^{\frac{2(n-1)}{n-4}}(\Gamma),\quad &n\geq5;\\
					H^2(\Omega)\hookrightarrow L^{p}(\Omega),\;\; &H^2(\Omega)\hookrightarrow H^{3/2}(\Gamma) \hookrightarrow L^{p}(\Gamma),\quad &\forall p\geq1, \ 1\le n\le 4.
				\end{array}\right.
			\end{eqnarray}
	\end{remark}
	By Sobolev imbedding theorem and trace-Sobolev imbedding theorem, there exist positive constants $\lambda_i$,  $i=0,1,2,3$, such that
	\begin{align*}
		&\|u\|^2\leq\lambda_0\|\Delta u\|^2;\ \  \|u\|^2_{\Gamma_1}\leq\lambda_1\|\Delta u\|^2;\ \  \|\eta^t\|^2\leq\lambda_2\|\Delta \eta^t\|^2;\ \ \|\eta^t\|^2_{\Gamma_1}\leq\lambda_3\|\Delta \eta^t\|^2.
	\end{align*}
Let $z(t)=(u(t),u_t(t),\eta^t)$. Under the above notations, the energy functional $E$ associated with the problem \eqref{P}-\eqref{initial-condition0} is denoted by
\begin{align}\label{energy-functional}
	E(z(t))=\frac{1}{2}||z(t)||^2_{\mathcal{H}}+\int_{\Gamma_1}F(u)d\Gamma.
\end{align}

\subsection{Global solutions}
We begin this section by presenting the solution definitions that we will use in this work.
\begin{definition}\rm \label{Definition-solution} A function $z=(u(t),u_t(t),\eta^{t})\in C([0,T];\mathcal{H})$ with the property $z(0)=z_0=(u_0,u_1,\eta_0)$, is said to be
	\begin{enumerate}
		\item[{\bf(R)}] {\it \underline{regular solution}} to problem \eqref{P}-\eqref{initial-condition0} on the interval $[0,T]$, iff
		\begin{itemize}
			\item $z=(u,u_t,\eta^t)\in W^{1,\infty}(0,T;\mathcal{H}),\;\; (u_t,\eta^t)\in W^{1,2}\left(0,T,L^2(\Gamma_1)\times L^2_g(\mathbb{R}^+;V)\right);$
			\item $\Delta^2 u(t)+\int_{0}^{\infty}g(\tau)\Delta^2\eta^t(\tau)d\tau \in L^2(\Omega)$ and $-\eta^t_s+u_t(t)\in L^2_g(\mathbb{R}^+;V)$ for almost all $t\in[0,T]$;
			\item Eq. $\eqref{P}_1$ is satisfied in $L^2(\Omega)$ and Eq. $\eqref{P}_2$ is satisfied in $L^2_g(\mathbb{R}^+;V)$ for almost all $t\in[0,T]$;
		\end{itemize}
		\item[{\bf(G)}] {\it \underline{generalized solution}} to problem \eqref{P}-\eqref{initial-condition0} on the interval $[0,T]$, iff there exists a sequence of strong solutions $\{z^n(t)\}=\{(u^n(t),u^n_t(t),\eta^{t,n})\}$ to problem \eqref{P}-\eqref{initial-condition0} with initial data $(u_{0n},u_{1n},\eta^{n}_0)$ instead of $(u_0,u_1,\eta_0)$ such that
		\begin{align*}\lim_{n\rightarrow \infty}\max_{\tau\in [0,t]}||z^n(\tau)-z(\tau)||_{\mathcal{H}}=0.
		\end{align*}
		\item[{\bf(W)}] {\it \underline{weak solution}} to problem \eqref{P}-\eqref{initial-condition0} on the interval $[0,T]$, if Eq. \eqref{P} is satisfied in the sense of distributions. The latter means that
		\begin{eqnarray}\left\{\begin{array}{l}\label{weak-solution}
				\displaystyle{(u_t(t),\omega)=(u_1,\omega)+\kappa\int_0^t(\Delta u,\Delta\omega)ds-\int_0^t\int_0^{\infty}g(\tau)\left(\Delta \eta^{s}(\tau),\Delta \omega\right)d\tau ds}\\
				\displaystyle{+\int_0^t(f(u(s))+a(s)u_s,\omega)_{\Gamma_1}ds=0,}\\
				\displaystyle{\int_0^t(\eta^s_t+\eta^s_{\tau},\theta)_{g,V}ds=\int_0^t(u_t,\theta)_{g,V}ds,}
			\end{array}\right.\end{eqnarray}
		holds for every $\omega\in V$, $\theta\in L^{2}_{g}(\mathbb{R}^{+};V)$, and for almost all $t\in [0,T]$.
	\end{enumerate}
\end{definition}
Under the previous definitions of solutions and Assumption \ref{A1}, we can state the theorem on the existence of global solutions for the problem \eqref{P}-\eqref{initial-condition0} as follows.
\begin{theorem}\label{theo-existence} Let $T>0$ be arbitrary.
	Under Assumption \ref{A1} the following statements hold.
	\begin{itemize}
		\item[{\bf I.}] {\bf [Regular solution]} For every $z(0)=z_0=(u_0,u_1,\eta_0)\in \mathcal{H}_1$, there exists unique regular solution $z=(u,u_t,\eta^{t})$ to problem \eqref{P}-\eqref{initial-condition0} satisfying Definition \ref{Definition-solution}-{\bf(R)}.
		\item[{\bf II.}] {\bf [Generalized (Weak solution)]} For every $z_0\in \mathcal{H}$ there exists unique generalized (weak) solution $z$ to problem \eqref{P}-\eqref{initial-condition0} satisfying Definition \ref{Definition-solution}-{\bf(G)} (Definition \ref{Definition-solution}-{\bf(W)}).
	\end{itemize}
\end{theorem}
\subsubsection*{Proof of the Theorem \ref{theo-existence}-${\bf I.}$}
The existence of a regular global solution of problem \eqref{P}-\eqref{initial-condition0} is classical, through the Galerkin method combined with Aubin-Lions compactness arguments. Indeed, let $\{\omega_j\}_{j\in\mathbb{N}}$
be the basis of $V$, be the orthonormal basis of $L^2(\Omega)$. Define $\theta_j=h_j\frac{\omega_j}{\|\omega_j\|_{V}},j=1,2,\cdots$, where $\{h_{j}\}_{j\in\mathbb{N}}$ is an orthonormal basis of $L^{2}_{g}(\mathbb{R}^+)\cap C_{0}^{\infty}(\mathbb{R}^+)$, then $\{\theta_{j}\}_{j\in\mathbb{N}}$ is a smooth orthonormal basis of $L^{2}_{g}(\mathbb{R}^+;V)$. Define $V_m=Span\{\omega_1,\omega_2,\cdots,\omega_m\}$ and
$Q_m=Span\{\theta_1,\theta_2,\cdots,\theta_m\}$.

Then by theory of ordinary differential equation, we know that
$u^m(t)=\sum\limits_{j=1}^{m}a_{mj}(t)\omega_{j}$ and $\eta^{t,m}(s)=\sum\limits_{j=1}^{m}b_{mj}(t)\theta_{j}(s)$ is the local solution of the following approximation system
\begin{eqnarray}\left\{\begin{array}{l}\label{aprox-equat}
		\displaystyle{
			<u^m_{tt},\omega_j>+\kappa(\Delta u^m(t),\Delta \omega_j)+(\eta^{t,m},\omega_j)_{g,V}+(f(u^m)+a(t)u^m_t,\omega_j)_{\Gamma_1}=0,}\\
		\displaystyle{(\eta^{t,m}_t+\eta^{t,m}_s,\theta_j)_{g,V}=(u^{m}_t,\theta_j)_{g,V},}
	\end{array}\right.\end{eqnarray}
for all $u^m(0)=u_0^{m}$, $u^m_t(0)=u_1^{m}$ and $\eta^{0,m}=\eta_0^m$, where $u_0^{m}\rightarrow u_0$ in $V$, $u_1^{m}\rightarrow u_1$ in $L^2(\Omega)$ and $\eta_0^m\rightarrow \eta_0$ in $L^2_{g}(\mathbb{R}^+;V).$
Next, our first a priori estimate allows us to extend the local solution to the entire interval $[0,T )$. The following estimates in this section will be used to prove the existence of a regular solution to \eqref{P}-\eqref{initial-condition0}. Without loss of generality, we will use the same parameter $C$ to denote different positive constants that will appear in the calculations, but we will also specify their dependence on time and initial data wherever necessary.
\subsubsection*{\it \underline{The first estimate.}}
We first consider the approximate system \eqref{aprox-equat} with
\begin{align}\label{convergence-initial}
	z^m(0)=(u^m_0 ,u^m_1,\eta^{m}_0)\rightarrow (u_0, u_1,\eta_0)=z_0\quad \mbox{strongly in}\quad \mathcal{H}.
\end{align}
Considering $\omega_j=u^m_t$ in $\eqref{aprox-equat}_1$ and $\theta_j=\eta^{t,m}$ in $\eqref{aprox-equat}_2$ we infer
\begin{eqnarray}\label{first-001}
	\frac{d}{dt}E(z^m(t))+a(t)\|u^m_t(t)\|^2_{\Gamma_1}+\left(\eta^{t,m}_s,\eta^{t,m}\right)_{g,V}=0,
\end{eqnarray}
where $E(z^m(t))$ is the energy functional \eqref{energy-functional} for Galerkin solutions $z^m$. From Assumption \eqref{hyp_M}, we have
\begin{align}\label{first-0}
	a(t)\|u^m_t(t)\|^2_{\Gamma_1}\ge m_0\|u^m_t(t)\|^2_{\Gamma_1}.
\end{align}
And by integration by parts and using Assumption \eqref{kernel_g'}, we obtain
\begin{align}\label{first-01}
	\left(\eta^{t,m}_s,\eta^{t,m}\right))_{g,V}=-\frac{1}{2}\int_0^{\infty}g'(\tau)\|\Delta \eta^{t,m}(\tau)\|^2d\tau\ge \frac{\alpha_2}{2}\|\eta^t\|^2_{L^2_g(\mathbb{R}^+;V)}.
\end{align}
Now, integrating \eqref{first-001} over $(0,t)$ and using \eqref{first-0}-\eqref{first-01}, we deduce
\begin{eqnarray}\label{first-1}
	E(z^m(t))+m_0\int_0^t\|u^m_t(\tau)\|^2_{\Gamma_1}d\tau+\frac{\alpha_2}{2}\int_0^{t}\|\eta^{\tau,m}\|^2_{L^2_g(\mathbb{R}^+;V)}d\tau\le E(z^m(0)).
\end{eqnarray}
From Assumption \eqref{diff2} and embedding $V\hookrightarrow L^2(\Gamma_1)$, we get
\begin{align}\label{first-2}
	\int_{\Gamma_1}F(u^m)d\Gamma\ge -C_fmeas(\Gamma_1)-\frac{c_f}{2}\|u(t)\|^2_{\Gamma_1}\ge -C_fmeas(\Gamma_1)-\frac{c_f\lambda_1}{2}\|\Delta u(t)\|^2.
\end{align}
Using that $c_f<\frac{\kappa}{\lambda_1}$ and defining $\varrho=1-\frac{c_f\lambda_1}{\kappa}$, we obtain from \eqref{first-2} that
\begin{eqnarray}\label{first-3}
	E(z^m(t))&\ge& -C_fmeas(\Gamma_1)+\frac{1}{2}\|u^m_t(t)\|^2+\varrho\frac{\kappa}{2}\|\Delta u^m(t)\|^2+\frac{1}{2}\|\eta^t\|^2_{L^2_g(\mathbb{R}^+;V)}\nonumber\\
	&\ge& -C_fmeas(\Gamma_1)+\frac{\varrho}{2}||z^m(t)||^2_{\mathcal{H}}.
\end{eqnarray}
Combining \eqref{first-1} and \eqref{first-3}, and using convergence \eqref{convergence-initial}, we get
\begin{eqnarray}\label{first-4}
	&&||z^m(t)||^2_{\mathcal{H}}+\frac{2m_0}{\varrho}\int_0^t\|u^m_t(\tau)\|^2_{\Gamma_1}d\tau+\frac{\alpha_2}{\varrho}\int_0^{t}\|\eta^{\tau,m}\|^2_{L^2_g(\mathbb{R}^+;V)}d\tau\nonumber\\
	&&\quad\le \frac{2(C_fmeas(\Gamma_1)+E(z^m(0)))}{\varrho}\le C,\quad \forall t\in (0,t_m),
\end{eqnarray}
where $C:=C(||z_0||_{\mathcal{H}}) > 0$ is a constant depending on weak initial data. Moreover, estimate \eqref{first-4} allows us to extend the local solution of the approximate problem to whole interval $[0,T),$ for any given $T>0$. Which implies that
\begin{align}
	(z^m)\quad &\mbox{is bounded in}\quad L^{\infty}(0,T;\mathcal{H}),\label{bounded-1}\\
	(u^m_t)\quad &\mbox{is bounded in}\quad L^{2}(0,T;L^2(\Gamma_1)),\label{bounded-2}\\
	(\eta^{t,m})\quad &\mbox{is bounded in}\quad L^{2}(0,T;L^2_g(\mathbb{R}^+;V)).\label{bounded-3}
\end{align}
\subsubsection*{\it \underline{The second estimate.}}
Now let's consider the approximate problem \eqref{aprox-equat} with initial conditions
\begin{align}
	\label{conv-initial-condition-2}
	z^m(0)=(u^m_0 ,u^m_1,\eta^{m}_0)\rightarrow (u_0, u_1,\eta_0)=z_0\quad \mbox{strongly in}\quad \mathcal{H}_1.
\end{align}
{\it In what follows, in the development of the calculations for simplicity we will omit the index $m.$} Let us fix $t,\epsilon$ such that $\epsilon<T-t$. Taking the difference of \eqref{aprox-equat} with $t=t+\epsilon$ and $t=t$ and replacing $\omega_j=u_t(t+\epsilon)-u_t(t)$ in $\eqref{aprox-equat}_1$ and $\theta_j=\eta^{t+\epsilon}-\eta^{t}$ in $\eqref{aprox-equat}_2$, we obtain
\begin{eqnarray}
	&&  \label{est3-1}\frac{1}{2}\frac{d}{dt}\mathcal{E}_{\epsilon}(t)+\underbrace{\left(\eta^{t+\epsilon}_s-\eta^{t}_s,\eta^{t+\epsilon}-\eta^{t}\right)_{g,V}}_{J_1}+\underbrace{M(\|u(t+\epsilon)\|^2_{\Gamma_1})\|u_t(t+\epsilon)-u_t(t)\|^2_{\Gamma_1}}_{J_2}\nonumber\\
	&&=-\underbrace{\left[\,M(\|u(t+\epsilon)\|^2_{\Gamma_1})-M(\|u(t)\|^2_{\Gamma_1})\,\right]\int_{\Gamma_1}u_t(t)\left[u_t(t+\epsilon)-u_t(t)\right]d\Gamma}_{J_3}\nonumber\\
	&&\quad-\underbrace{\int_{\Gamma_1}\left[\,f(u(t+\epsilon)-f(u(t))\,\right]\left[u_t(t+\epsilon)-u_t(t)\right]d\Gamma}_{J_4},
\end{eqnarray}
where
\begin{align}\label{definitionE}\mathcal{E}_{\epsilon}(t):=\|u_t(t+\epsilon)-u_t(t)\|^2+\kappa\|\Delta u(t+\epsilon)-\Delta u(t)\|^2+\|\eta^{t+\epsilon}-\eta^{t}\|^2_{L^2_g(\mathbb{R}+;V)}.\end{align}
Our objective now is to estimate the terms on the right-hand side of \eqref{est3-1}, in order to apply Gronwall's lemma. But first, note that terms $J_1$ and $J_2$ on the left side of \eqref{est3-1} are positive and can be estimated lower as follows: using \eqref{hyp_M}, we have
\begin{align}
	M(\|u(t+\epsilon)\|^2_{\Gamma_1})\|u_t(t+\epsilon)-u_t(t)\|^2_{\Gamma_1}\ge m_0\|u_t(t+\epsilon)-u_t(t)\|^2_{\Gamma_1}.
\end{align}
And, from Assumption \eqref{kernel_g'}, we get
\begin{eqnarray*}
	\left(\eta^{t+\epsilon}_s-\eta^{t}_s,\eta^{t+\epsilon}-\eta^{t}\right)_{g,V}&=&-\frac{1}{2}\int_0^{\infty}g'(\tau)\|\Delta \left(\eta^{t+\epsilon}(\tau)-\eta^{t}(\tau)\right)\|^2d\tau\\
	&\ge &\frac{\alpha_2}{2}\int_0^{\infty}g(\tau)\|\Delta \left(\eta^{t+\epsilon}(\tau)-\eta^{t}(\tau)\right)\|^2d\tau\\
	&=&\frac{\alpha_2}{2}\|\eta^{t+\epsilon}-\eta^{t}\|^2_{L^2_g(\mathbb{R}^+;V)}.
\end{eqnarray*}
Next, the terms $J_3$ and $J_4$ on the right side of \eqref{asympt-12} can be estimated as follows: using that $M\in C^1(\mathbb{R}^+)$, Mean Value Theorem, and embedding $V\hookrightarrow L^{2}(\Gamma_1)$, we get
\begin{eqnarray*}
	M(\|u(t+\epsilon)\|^2_{\Gamma_1})-M(\|u(t)\|^2_{\Gamma_1})&\le &C\left|\|u(t+\epsilon)\|^2_{\Gamma_1}-\|u(t)\|^2_{\Gamma_1}\right|\\
	&\le& C\|u(t+\epsilon)-u(t)\|_{\Gamma_1}\\
	&\le & C\|\Delta u(t+\epsilon)-\Delta u(t)\|.
\end{eqnarray*}
Hence, from the first estimate, Young's inequality, and \eqref{definitionE}, we have
\begin{eqnarray*}
	J_3&\le& C\|\Delta u(t+\epsilon)-\Delta u(t)\|\|u_t(t+\epsilon)-u_t(t)\|_{\Gamma_1}\\
	&\le& C\|\Delta u(t+\epsilon)-\Delta u(t)\|^2+\frac{m_0}{4}\|u_t(t+\epsilon)-u_t(t)\|^2_{\Gamma_1}\\
	&\le &C\mathcal{E}_{\epsilon}(t)+\frac{m_0}{4}\|u_t(t+\epsilon)-u_t(t)\|^2_{\Gamma_1}.
\end{eqnarray*}
Finally, from Assumption \eqref{diff}, H\"older inequality with $\frac{p}{2(p+1)}+\frac{1}{2(p+1)}+\frac{1}{2}=1$, embedding $V\hookrightarrow L^{2(p+1)}(\Gamma_1)$, and Young's inequality, we have
\begin{eqnarray*}
	J_4&\le& C_{f'}\int_{\Gamma_1}\left[1+(|u(t+\epsilon)|+|u(t)|)^{p}\right]|u(t+\epsilon)-u(t)||u_t(t+\epsilon)-u_t(t)|d\Gamma\\
	&\le& C\left[1+\|u(t+\epsilon)\|^{p}_{2(p+1),\Gamma_1}+\|u(t)\|^p_{2(p+1),\Gamma_1}\right]\|u(t+\epsilon)-u(t)\|_{2(p+1),\Gamma_1}\|u_t(t+\epsilon)-u(t)\|_{\Gamma_1}\\
	&\le & C\|\Delta u(t+\epsilon)-\Delta u(t)\|\|u_t(t+\epsilon)-u(t)\|_{\Gamma_1}\\
	&\le &C\mathcal{E}_{\epsilon}(t)+\frac{m_0}{4}\|u_t(t+\epsilon)-u_t(t)\|^2_{\Gamma_1}.
\end{eqnarray*}
Returning to \eqref{asympt-12}, we get
\begin{align}	\label{est3-5}
	\frac{d}{dt}\mathcal{E}_{\epsilon}(t)+m_0\|u_t(t+\epsilon)-u_t(t)\|^2_{\Gamma_1}+\alpha_2\|\eta^{t+\epsilon}-\eta^{t}\|^2_{L^2_g(\mathbb{R}^+;V)}\le C\mathcal{E}_{\epsilon}(t),\quad \forall t\in[0,T).
\end{align}
Applying Gronwall's lemma in \eqref{est3-5} we have
\begin{align}
	\label{est3-6} \mathcal{E}_{\epsilon}(t)+\int_0^t\left[\,\|u_t(s+\epsilon)-u_t(s)\|^2_{\Gamma_1}+\|\eta^{s+\epsilon}-\eta^{s}\|^2_{L^2_g(\mathbb{R}^+;V)}\,\right]ds\le \frac{e^{Ct}\mathcal{E}_{\varepsilon}(0)}{\min\{1,m_0,\alpha_2\}},
\end{align}
for all $t\in[0,T)$.
{\it From now on we return with the index $m$.} \\
Dividing \eqref{est3-6} by $\epsilon^2$ and letting $\epsilon\rightarrow 0$, we obtain
\begin{align}\label{est3-7}
	||z_t(t)||^2_{\mathcal{H}}+\int_0^t\left[\|u_{tt}(s)\|^2_{\Gamma_1}+\|\eta^{s}_t\|^2_{L^2_g(\mathbb{R}^+;V)}\right]ds\le \frac{e^{Ct}||z^m_t(0)||^2_{\mathcal{H}}}{\min\{1,m_0,\alpha_2\}},
\end{align}
$\forall m\in \mathbb{N}$ and $\forall t\in [0,T)$.
Since $$||z^m_t(0)||^2_{\mathcal{H}}=\|u^m_{tt}(0)\|^2+\|\Delta u^m_t(0)\|+\|\eta^{0,m}_t\|^2_{L^2_g(\mathbb{R}^{+};V)},$$
using that $z_0\in \mathcal{H}_1$, there exists $C>0$ such that
\begin{align*}
	\|\Delta u_{1m}\|^2+\|\eta^{0,m}_t\|^2_{L^2_g(\mathbb{R}^+;V)}\le C,\quad \forall m\in \mathbb{N},\,\,\forall t\in [0,T).
\end{align*}
Thus, it remains to obtain an estimate for $\|u^m_{tt}(0)\|$. Indeed, considering $t=0$ in \eqref{aprox-equat} and substituting $\omega_j=u^m_{tt}(0)$ in $\eqref{aprox-equat}_1$, it results that
\begin{eqnarray}\label{utt0}
	&&  \|u^m_{tt}(0)\|^2+\kappa(\Delta u^m(0),\Delta u^m_{tt}(0))+(\eta^{0,m},u^m_{tt}(0))_{g,V}\nonumber\\
	&&\quad +(f(u^m(0))+a(u^m(0))u^m_t(0),u^m_{tt}(0))_{\Gamma_1}=0.
\end{eqnarray}
Using integration by parts and the compatibility condition \eqref{comp-cond}, we have
\begin{eqnarray*}
	&&\kappa(\Delta u^m(0),\Delta u^m_{tt}(0))+(\eta^{0,m},u^m_{tt}(0))_{g,V}\\
	&&=\kappa\int_{\Omega}\Delta u^m(0)\Delta u^m_{
		tt}(0)dx+\int_0^{\infty}g(\tau)\int_{\Omega}\Delta \eta^{0,m}(\tau)\Delta u^m_{tt}(0)dxd\tau\\
	&&=-\int_{\Gamma}\Big{(}\kappa\frac{\partial(\Delta u^m(0))}{\partial\nu}+\int_{0}^{\infty}g(\tau)\frac{\partial(\Delta \eta^{0,m}(\tau))}{\partial\nu}d\tau\Big{)}u^m_{tt}(0)d\Gamma\\
	&&\quad +\kappa\int_{\Omega}\Delta^2 u^m(0)u^{m}_{tt}(0)dx+\int_0^{\infty}g(\tau)\int_{\Omega}\Delta^2 \eta^{0,m}(\tau)u^{m}_{tt}(0)dxd\tau\\
	&&=-\int_{\Gamma_1}(f(u^m(0))+a(u^m(0))u^m_t(0))u^m_{tt}(0)d\Gamma\\
	&&\quad +\kappa\int_{\Omega}\Delta^2 u^m(0)u^{m}_{tt}(0)dx+\int_{\Omega}\int_0^{\infty}g(\tau)\Delta^2 \eta^{0,m}(\tau)d\tau u^{m}_{tt}(0)dx.
\end{eqnarray*}
Replacing the above equality into \eqref{utt0}, we obtain that
\begin{eqnarray*}
	\|u^m_{tt}(0)\|^2\le \left[\,\kappa\|\Delta^2u^m(0)\|+\int_0^{\infty}g(\tau)\|\Delta^2 \eta^{0,m}(\tau)\|d\tau\,\right]\|u^m_{tt}(0)\|.
\end{eqnarray*}
Thus, there exists $C>0$ such that
\begin{align}\label{est-utt0}
	\|u^m_{tt}(0)\|\le C,\quad \forall m\in \mathbb{N}.\end{align}
Returning to \eqref{est3-7}, we get
\begin{align}\label{est3-8}
	||z_t(t)\|^2_{\mathcal{H}}\equiv \|u^m_{tt}(t)\|^2+\kappa\|\Delta u^m_t(t)\|^2+\|\eta^{t,m}_t\|^2_{L^2_g(\mathbb{R}^+;V)}\le C,\quad \forall m\in \mathbb{N},\,\, \forall t\in [0,T).
\end{align}
From \eqref{est3-8}, we have
\begin{align}
	(z^m_t)\quad &\mbox{is bounded in}\quad L^{\infty}(0,T;\mathcal{H}),\label{bounded-4}\\
	(u^m_{tt})\quad &\mbox{is bounded in}\quad L^{2}(0,T;L^2(\Gamma_1)),\label{bounded-5}\\
	(\eta^{t,m}_t)\quad &\mbox{is bounded in}\quad L^{2}(0,T;L^2_g(\mathbb{R}^+;V)).\label{bounded-6}
\end{align}
\subsubsection*{\it \underline{Passage to the limit and existence of regular solution.}}
From \eqref{bounded-1}-\eqref{bounded-3} and \eqref{bounded-4}-\eqref{bounded-6} we obtain a subsequence $z^k=({u^k},u^k_t,\eta^{t,k})$ of  $z^m=({u^m},u^m_t,\eta^{t,m})$ such that
\begin{align}
	z^k\rightharpoonup z\quad &\mbox{weak-star in}\quad L^{\infty}(0,T;\mathcal{H}),\label{conv-1}\\
	z^k_t\rightharpoonup z_t\quad &\mbox{weak-star in}\quad L^{\infty}(0,T;\mathcal{H}),\\
	u^k\rightharpoonup u\quad &\mbox{weak in}\quad L^{2}(0,T;L^{2}(\Gamma_1)),\\
	u^k_t\rightharpoonup u_t\quad &\mbox{weak in}\quad L^{2}(0,T;L^{2}(\Gamma_1)),\\
	\eta^{t,k}\rightharpoonup \eta^t\quad &\mbox{weak in }\quad L^{2}(0,T;L^2_g(\mathbb{R}^+;V)),\\
	\eta^{t,k}_t\rightharpoonup \eta_t^t\quad &\mbox{weak in }\quad L^{2}(0,T;L^2_g(\mathbb{R}^+;V)).\label{conv-2}
\end{align}
With the convergences \eqref{conv-1}-\eqref{conv-2} we can use Lions-Aubin lemma to get necessary compactness in order to pass the limit in the approximate problem \eqref{aprox-equat} to obtain a regular function $z=(u,u_t,\eta^t)$ such that
\begin{align*}
	&u_{tt}+\kappa\Delta ^2u+\int_{0}^{\infty}g(\tau)\Delta^2\eta^t(\tau)d\tau=0\quad \mbox{in}\quad L^{\infty}(0,T;L^2(\Omega)),\\
	&\eta^t_t+\eta^t_s=u_t \quad \mbox{in}\quad L^{\infty}(0,T;L^2_g(\mathbb{R}^+;V)).
\end{align*}

\subsubsection*{Proof of the Theorem \ref{theo-existence}-${\bf II.}$}
Now, let's show the existence of a generalized (weak) solution to the problem \eqref{P}-\eqref{initial-condition0}. Indeed, let $z_0=(u_0,u_1,\eta_0)\in \mathcal{H}=\overline{\mathcal{H}_1}^{V\times L^{2}(\Omega)\times L^2_g(\mathbb{R}^+;V)}$. Then, since $\mathcal{H}_1$ is dense in $\mathcal{H}$ there exists a sequence $z^{\mu}_{0}=(u^{\mu}_{0},u^{\mu}_{1},\eta^{\mu}_0)\in \mathcal{H}_1$ such that
\begin{align}\label{weak-est0}
	||z^{\mu}_0-z_0||_{\mathcal{H}}\rightarrow_{\mu\rightarrow +\infty} 0.
\end{align}
From Theorem \ref{theo-existence}, for each $\mu\in \mathbb{N}$, there exists a regular solution $z^{\mu}(t)=(u^{\mu}(t),u^{\mu}_t(t),\eta^{t,\mu })$ that satisfies problem \eqref{P}-\eqref{initial-condition0}.

Multiplying the Eq, $\eqref{P}_1$ by $u^{\mu}_t$ and Eq. $\eqref{P}_2$ by $\eta^{t,\mu}$, integrating over $\Omega$, and using arguments similar to those used in the first estimate, we infer
\begin{align}\label{weak-est-1}
	||z^{\mu}(t)||^2_{\mathcal{H}}\le C,
\end{align}
where $C>0$ is a constant independent of $t\in[0,T]$ and $\mu\in \mathbb{N}$.

On the other hand, let $z^{\mu}(t)=(u^{\mu}(t),u^{\mu}_t(t),\eta^{t,\mu})$ and $z^{\sigma}(t)=(u^{\sigma}(t),u^{\sigma}_t(t),\eta^{t,\sigma})$ regular solutions of \eqref{P}-\eqref{initial-condition0}. Defining $w^{\sigma,\mu}=u^{\mu}-u^{\sigma}$, $\zeta^{t(\sigma,\mu)}=\eta^{t,\mu}-\eta^{t,\sigma}$,  $\widehat{f}(w^{\sigma,\mu})=f(u^{\mu})-f(u^{\sigma})$, and $\widehat{M}(w^{\sigma,\mu})=M(\|u^{\mu}(t)\|^2_{\Gamma_1})u^{\mu}_t-M(\|u^{\sigma}(t)\|^2_{\Gamma_1})u^{\sigma}_t$, then $z^{\mu}-z^{\sigma}=(w^{\sigma,\mu},\zeta^{t(\sigma,\mu)})$ satisfies the following problem
\begin{eqnarray}\left\{\begin{array}{l}\label{PD}
		w^{\sigma,\mu}_{tt}+\kappa\Delta^2 w^{\sigma,\mu}+\int_0^{\infty}g(\tau)\Delta^2\zeta^{t(\sigma,\mu )}(\tau)d\tau=0,\quad \mbox{in}\quad \Omega \times \mathbb{R}^+,\\
		\zeta_t^{t(\sigma,\mu)}=-\zeta_s^{t(\sigma,\mu)}+w^{\sigma,\mu}_t,\quad \mbox{in}\quad \Omega \times \mathbb{R}^+\times \mathbb{R}^+,
	\end{array}\right.\end{eqnarray}	
with  boundary condition
\begin{eqnarray}
	\left\{\begin{array}{l}
		\label{boundary-condition2}
		w^{\sigma,\mu}=\frac{\partial w^{\sigma,\mu}}{\partial \nu}=0\quad \mbox{on}\quad \Gamma_0\times\mathbb{R}^+\quad\mbox{and}\quad \zeta^{t(\sigma,\mu)}=\frac{\partial \zeta^{t(\sigma,\mu)}}{\partial \nu}=0\quad \mbox{on}\quad \Gamma_0\times[\mathbb{R}^+]^2,\\
		\kappa\Delta w^{\sigma,\mu}+\int_{0}^{\infty}g(\tau)\Delta \zeta^{t(\sigma,\mu)}(\tau)d\tau=0\quad\mbox{on}\quad \Gamma_1\times\mathbb{R}^+\times \mathbb{R}^+,\\
		\kappa\frac{\partial(\Delta w^{\sigma,\mu})}{\partial \nu}+\int_{0}^{\infty}g(\tau)\frac{\partial(\Delta \zeta^{t(\sigma,\mu)}(\tau))}{\partial\nu}d\tau=\widehat{f}(w^{\sigma,\mu})+\widehat{M}(w^{\sigma,\mu})\quad\mbox{on}\quad \Gamma_1\times\mathbb{R}^+\times \mathbb{R}^+,
	\end{array}\right.
\end{eqnarray}
and initial conditions
\begin{eqnarray}
	\left\{\begin{array}{l}
		\label{initial-condition11}
		w^{\sigma,\mu}(x,0)=w^{\sigma,\mu}_{0}(x),\quad w^{\sigma,\mu}_t(x,0)=w^{\sigma,\mu}_1(x),\quad x\in  \Omega,\\
		\zeta^{0(\sigma,\mu)}(x,s)=\zeta^{\sigma,\mu}_0(x,s),\quad\mbox{in}\quad \Omega\times \mathbb{R}^+.
	\end{array}\right.
\end{eqnarray}
Multiplying the Eq. $\eqref{PD}_1$ by $w^{\sigma,\mu}_t$ and Eq. $\eqref{PD}_2$ by $\zeta^{t(\sigma,\mu)}$, and integrating over $\Omega$, we obtain
\begin{align}	\label{asympt-12}\frac{1}{2}\frac{d}{dt}||z^{\mu}(t)-z^{\sigma}(t)||^2_{\mathcal{H}}+\left(\zeta^{t(\sigma,\mu)}_s,\zeta^{t(\sigma,\mu)}\right)_{g,V}+M(\|u^{\mu}(t)\|_{\Gamma_1}^2)\|w^{\sigma,\mu}_t(t)\|^2_{\Gamma_1}=L_1+L_2,
\end{align}
where
\begin{eqnarray*}
	L_1&=&-\left[\,M(\|u^{\mu}(t)\|^2_{\Gamma_1})-M(\|u^{\sigma}(t)\|^2_{\Gamma_1})\,\right]\int_{\Gamma_1}u^{\sigma}_tw^{\sigma,\mu}_td\Gamma,\\
	L_2&=&-\int_{\Gamma_1}\widehat{f}(w^{\sigma,\mu})w^{\sigma,\mu}_td\Gamma.
\end{eqnarray*}
Next, the two terms of the sum on the left side of the equality \eqref{asympt-12} can be reduced as follows. From Assumption \eqref{kernel_g'}, we get
\begin{eqnarray*}
	\left(\zeta^{t(\sigma,\mu)}_s,\zeta^{t(\sigma,\mu)}\right)_{g,V}&=&-\frac{1}{2}\int_0^{\infty}g'(\tau)\|\Delta \zeta^{t(\sigma,\mu)}(\tau)\|^2d\tau\\
	&\ge &\frac{\alpha_2}{2}\int_0^{\infty}g(\tau)\|\Delta \zeta^{t(\sigma,\mu)}(\tau)\|^2d\tau\\
	&=&\frac{\alpha_2}{2}\|\zeta^{t(\sigma,\mu)}\|^2_{L^2_g(\mathbb{R}^+;V)}.
\end{eqnarray*}
And, using that $M(\|u^{\mu}(t)\|^2_{\Gamma_1})\ge m_0>0$, it is straightforward that
\begin{align*}
	M(\|u^{\mu}(t)\|^2_{\Gamma_1})\|w^{\sigma,\mu}_t(t)\|^2_{\Gamma_1}\ge m_0\|w^{\sigma,\mu}_t(t)\|^2_{\Gamma_1}.
\end{align*}
Now, the terms $L_1$ and $L_2$ on the right side of \eqref{asympt-12} can be estimated as follows. Using that $M\in C^1(\mathbb{R}^+)$ and embedding $V\hookrightarrow L^{2}(\Gamma_1)$, we get
\begin{eqnarray*}
	M(\|u^{\mu}(t)\|^2_{\Gamma_1})-M(\|u^{\sigma}(t)\|^2_{\Gamma_1})\le C\left|\|u^{\mu}(t)\|^2_{\Gamma_1}-\|u^{\sigma}(t)\|^2_{\Gamma_1}\right|\le C\|w^{\sigma,\mu}(t)\|_{\Gamma_1}.
\end{eqnarray*}
Then, from Young's inequality, we have
\begin{eqnarray*}
	L_1\le C\|w^{\sigma,\mu}(t)\|_{\Gamma_1}\|w^{\sigma,\mu}_t(t)\|_{\Gamma_1}\le C\|w^{\sigma,\mu}(t)\|^2_{\Gamma_1}+\frac{m_0}{4}\|w^{\sigma,\mu}_t(t)\|^2_{\Gamma_1}.
\end{eqnarray*}
Finally, from Assumpton \eqref{diff}, H\"older inequality with $\frac{p}{2(p+1)}+\frac{1}{2(p+1)}+\frac{1}{2}=1$, embedding $V\hookrightarrow L^{2(p+1)}(\Gamma_1)$, and Young's inequality, we have
\begin{eqnarray*}
	L_2&\le& C_{f'}\int_{\Gamma_1}\left[1+(|u^{\mu}|+|u^{\sigma}|)^{p}\right]|w^{\sigma,\mu}||w^{\sigma,\mu}_t|d\Gamma\\
	&\le& C\left[1+\|u^{\mu}(t)\|^{p}_{2(p+1),\Gamma_1}+\|u^{\sigma}(t)\|^p_{2(p+1),\Gamma_1}\right]\|w^{\sigma,\mu}(t)\|_{2(p+1),\Gamma_1}\|w^{\sigma,\mu}_t(t)\|_{\Gamma_1}\\
	&\le& C\|w^{\sigma,\mu}(t)\|_{2(p+1),\Gamma_1}\|w^{\sigma,\mu}_t(t)\|_{\Gamma_1}\\
	&\le &C\|w^{\sigma,\mu}(t)\|^2_{2(p+1),\Gamma_1}+\frac{m_0}{4}\|w^{\sigma,\mu}_t(t)\|^2_{\Gamma_1}.
\end{eqnarray*}
Returning to \eqref{asympt-12}, we get
\begin{eqnarray}	\label{asympt-21}
	&&	\frac{d}{dt}||z^{\mu}(t)-z^{\sigma}(t)||^2_{\mathcal{H}}+\alpha_2\|\zeta^{t(\sigma,\mu)}\|^2_{L^2_g(\mathbb{R}^+;V)}+m_0\|w^{\sigma,\mu}_t(t)\|^2_{\Gamma_1}\nonumber\\
	&&\le C\left[\,\|w^{\sigma,\mu}(t)\|^2_{\Gamma_1}+\|w^{\sigma,\mu}(t)\|^2_{2(p+1),\Gamma_1}\,\right].
\end{eqnarray}
Now, using embeddings $V\hookrightarrow L^{2(p+1)}(\Gamma_1)\hookrightarrow L^2(\Gamma_1)$, it follows from \eqref{asympt-21} that
\begin{align}	\label{asympt-31}
	\frac{d}{dt}||z^{\mu}(t)-z^{\sigma}(t)||^2_{\mathcal{H}}\le C||z^{\mu}(t)-z^{\sigma}(t)||^2_{\mathcal{H}}.
\end{align}
Applying Gronwall's Lemma in \eqref{asympt-31}, we obtain
\begin{align}\label{weak-est-2}
	||z^{\mu}(t)-z^{\sigma}(t)||^2_{\mathcal{H}}\le e^{Ct}||z^{\mu}(0)-z^{\sigma}(0)||^2_{\mathcal{H}},\quad \forall t\in[0,T].
\end{align}
Using strong convergence \eqref{weak-est0}, it follows from \eqref{weak-est-2} that
\begin{align}\label{weak-est-22}
	||z^{\mu}(t)-z^{\sigma}(t)||^2_{\mathcal{H}}\le e^{CT}||z^{\mu}_0-z^{\sigma}_0||^2_{\mathcal{H}}\longrightarrow_{\mu,\sigma\rightarrow \infty} 0,\quad \forall t\in[0,T],
\end{align}
where $C=C(||z^{\mu}(0)||_{\mathcal{H}},||z^{\sigma}(0)||_{\mathcal{H}})$.

This proves that $\{z^{\mu}\}_{\mu\in \mathbb{N}}$ is a Cauchy's sequence of $C([0, T]; \mathcal{H})$. Therefore there exists a function $z\in C([0, T]; \mathcal{H})$ such that $z^{\mu}\rightarrow z$ in $ C([0, T]; \mathcal{H})$, that is
\begin{equation}\label{a1}
	\lim_{\mu\rightarrow +\infty}\max_{\tau\in [0,T]}||z^{\mu}(\tau)-z(\tau)||^2_{\mathcal{H}}=0.
\end{equation}
\begin{remark}
	Due to the regularity, the boundary condition \eqref{boundary-condition} makes  sense and implies that the compatibility condition \eqref{comp-cond} holds for $z^{\mu}(t)=(u^{\mu}(t),u^{\mu}_t(t),\eta^{t,\mu})$, for any $t\ge 0$.
	This shows that, for any $t\ge 0$, $z(t)$ is the $V\times L^2(\Omega)\times L^2_g(\mathbb{R}^+;V)$-limit of a sequence $z^{\mu}\in \mathcal{H}_1,$ that is $z(t)\in \mathcal{H}$.
\end{remark}	
\medskip
Note that the function $z(t)$ satisfies Definition \ref{Definition-solution}-${\bf (G})$. This guarantees the existence of a generalized solution for problem \eqref{P}-\eqref{initial-condition0}. By direct calculations we can show that \eqref{weak-solution} holds for the generalized solution $z(t)$, i.e., $z$ is also a weak solution.

\subsubsection*{\it \underline{Uniqueness}}
Let $z^1(t)=(u^1(t),u^1_t(t),\eta^t_1)$ and $z^2=(u^2(t),u^2_t(t),\eta^t_2)$ be two regular solutions of problem \eqref{P}-\eqref{initial-condition0} with respect to initial data $z^1(0)=(u^1(0),u^1_t(0),\eta^0_1),z^2(0)=(u^2(0),u^2_t(0),\eta^0_2)$, respectively, and set $w =u^1-u^2$ and $\zeta^t=\eta^t_1-\eta^t_2$. Proceeding analogously as in the proof of the generalized solution, using $u^1$ in the place of $u^{\mu}$, $u^2$ in the place of $u^{\sigma}$, $\eta^{t,\mu}$ in the place of $\eta^t_1$, and $\eta^{t,\sigma}$ in the place of $\eta^t_2$, we obtain the following inequality, similar to inequality \eqref{weak-est-2}
\begin{align*}
	||z^{1}(t)-z^{2}(t)||^2_{\mathcal{H}}\le e^{CT}||z^{1}(0)-z^{2}(0)||^2_{\mathcal{H}},\quad \forall t\in[0,T],
\end{align*}
where $C=C(||z^1(0)||_{\mathcal{H}},||z^2(0)||_{\mathcal{H}})$. This ensures the continuous dependence of regular solutions on the initial data in $\mathcal{H}.$ The same conclusion holds for generalized (weak) solutions by using density arguments. In particular for $z^1(0)=z^2(0)$, problem \eqref{P}-\eqref{initial-condition0} has a unique regular or generalized (weak) solution.

\medskip
This completes the proof of the Theorem \ref{theo-existence}.

\subsubsection*{Generation of a dynamical system}
The well-posedness of the problem \eqref{P}-\eqref{initial-condition0} given by Theorem \ref{theo-existence} implies that the evolution operator $S_t:\mathcal{H}\rightarrow \mathcal{H}$ defined by the formula
\begin{align}\label{semigroup}
	S_tz_0=z(t),\ t\geq0,
\end{align}
where $z(t)=(u(t),u_t(t),\eta^t)$ is the unique generalized (weak) solution of problem \eqref{P}-\eqref{initial-condition0} with initial value $z_0=(u_0,u_1,\eta_0)$, defines a nonlinear continuous semigroup which is locally Lipschitz continuous on the phase space $\mathcal{H}$. Hence, $(\mathcal{H},S_t)$ constitutes a dynamical systems that will describe the long-time dynamics of problem \eqref{P}-\eqref{initial-condition0}.

\section{Attractors for abstract dynamical systems}
This Section is dedicated to the presentation of some abstract results of the theory of connected dynamical systems with the existence of global attractors and their properties. This theory can be found in Chueshov and Lasiecka's books    \cite{chueshov-white,chueshov-yellow}.

\begin{definition}\rm
	A dynamical system is a pair $(X,S_t)$ of a complete metric space $X$ and a family of continuous mappings $S_t:X \to X$ such that $y\mapsto S_ty$ is continuous in $X$ for every $y\in X$ and the semigroup property is satisfied:
	$$S_0=I,\quad S_{t+s}=S_t\circ S_s,\quad  \forall s,t\geq0.$$
	$S_t$ is called evolution semigroup (or evolution operator).
\end{definition}
\begin{definition}\rm
	Let $(X,S_t)$ be a dynamical system given by a continuous semigroup $S_t$ on a complete metric space $X$. Then a {\bf global attractor} for $S_t$ is a compact set $\mathfrak{A}\subset X$ that is fully invariant and uniformly attracting, that is, $S_t\mathfrak{A}=\mathfrak{A}$ for all $t\ge 0$ and for every bounded subset $B\subset X,$
	$$
	\mbox{dist}_{X}(S_tB,\mathfrak{A}) = \sup_{x\in S_tB}\inf_{y\in
		\mathfrak{A}}d_{X}(x,y) \to 0 \quad \mbox{as} \quad t \to \infty.
	$$
\end{definition}

\begin{definition}\rm
	A bounded set $\mathcal{B} \subset X$ is an {\bf absorbing set} for $S_t$ if for any bounded set $B\subset X$, there exists $t_B\geq 0$ such that
	$$
	S_tB\subset \mathcal{B} , \quad \forall \, t \geq t_B ,
	$$
	which defines $(X,S_t)$ as a {\bf dissipative} dynamical system.
\end{definition}

\begin{definition}\rm
	A dynamical system $(X,S_t)$ is {\bf asymptotically smooth} in $X$ if for any bounded positive invariant set $B \subset X$, there exists a compact set $K \subset \overline{B}$, such that
	$$
	\mbox{dist}_{X}(S_tB,K) = 0 \quad \mbox{as} \quad t \to \infty .
	$$
\end{definition}

\begin{definition}\rm
	Let $\mathcal{N}$ be the set of stationary points of a dynamical system $(X,S_t)$:
	$$\mathcal{N}=\{v\in X:S_tv=v\;\;\mbox{for all}\;\;t\ge 0\}.$$
	We define the {\bf unstable manifold} $\mathrm{M}^u(\mathcal{N})$ emanating from set $\mathcal{N}$ as a set of all $y\in X$ such that there exists a full trajectory $\Upsilon=\{u(t):t\in \mathbb{R}\}$ with the properties
	$$u(0)=y\quad \mbox{and}\quad \lim_{t\rightarrow-\infty}\mbox{dist}_{X}(u(t),\mathcal{N})=0.$$
\end{definition}
\begin{definition}\rm
	Given a compact set $K$ in a metric space $X$, the {\bf fractal
		dimension} of $K$ is defined by
	$$
	\mbox{dim}_{f}^{X} K = \limsup_{\varepsilon \to 0} \frac{\ln
		n(K,\varepsilon)}{\ln(1/\varepsilon)},
	$$
	where $n(K,\varepsilon)$ is the minimal number of closed balls of
	radius $\varepsilon$ which covers $K$. Since the  Hausdorff
	dimension of a compact set does not exceed the fractal one (see
	e.g. \cite{hale}, Chapter 2), it is enough to prove the finiteness of the fractal dimension.
\end{definition}

\medskip
The next two theorems guarantee the existence of a compact global attractor and its geometric characterization for gradient, dissipative and asymptotically smooth systems.
\begin{theorem}[Theorem 2.3, \cite{chueshov-white}] \label{theo-chueshov}
Let $(X,S_t)$ be a dissipative dynamical system in a complete metric space $X$. Then $(X,S_t)$ possesses a compact global attractor $\mathfrak{A}$ if and only if $(X,S_t)$ is asymptotically smooth. 
\end{theorem}
\begin{theorem}[Theorem 7.5.6, \cite{chueshov-yellow}] \label{theo-chueshov-caracterization}
Let a dynamical system $(X,S_t)$ possess a compact global attractor $\mathfrak{A}$. Assume that there exists a strict Lyapunov function on $\mathfrak{A}$. Then $\mathfrak{A}=\mathcal{M}^z(\mathcal{N})$. Moreover the global attractor $\mathfrak{A}$ consists of full trajectories $\Upsilon=\{z(t):\,t\in \mathbb{R}\}$ such that
$$\lim_{t\rightarrow -\infty}\mbox{dist}_{X}(z(t),\mathcal{N})=0\quad \mbox{and}\quad \lim_{t\rightarrow +\infty}\mbox{dist}_{X}(z(t),\mathcal{N})=0.$$
\end{theorem}

\begin{Assumption}\label{assumption-quasi-stable}\rm
	Let $Y,Z,W$ be three reflexive Banach spaces with $Y$ compactly embedded in $Z$ and set $X=Y\times Z \times W$, where the case with trivial space $W = \{0\}$ is allowed. Consider the dynamical system $(X,S_t)$ given by an evolution
	operator
	\begin{equation}\label{inforce1}
		S_tz = (u(t),u_{t}(t),\theta(t)), \qquad z=(u_0, u_1, \theta_0)
		\in X ,
	\end{equation}
	where the functions $u$ and $\theta$ possess the regularity
	\begin{equation}\label{inforce2}
		u \in C(\mathbb{R}^{+}; Y)\cap C^{1}(\mathbb{R}^{+}; Z), \quad
		\theta \in C(\mathbb{R}^{+}; W).
	\end{equation}
 \end{Assumption}
 \begin{definition}\rm \label{Def-quasi-stable} A dynamical system of the form \eqref{inforce1} is said to be stable modulo compact terms on a set
	$B \subset Y$ if there exist a compact seminorm $n_Y$ on
	$Y$ and nonnegative scalar functions $a(t)$ and $c(t)$ locally
	bounded in $[0,\infty)$, and $b(t)\in L^1(\mathbb{R}^{+})$ with
	$\displaystyle\lim_{t\to \infty}b(t)=0,$ such that
	\begin{equation} \label{local-lips}
		\Vert S_tz^1 - S_tz^2 \Vert_{X}^2 \le a(t)\Vert z^1 -
		z^2 \Vert_{X}^2,
	\end{equation}
	and
	\begin{equation} \label{stabili}
		\Vert S_tz^1 - S_tz^2 \Vert_{X}^2 \le b(t)\Vert z^1 -
		z^2 \Vert_{X}^2 + c(t) \sup_{0<s<t}\left[
		n_Y(u^1(s)-u^2(s)) \right]^2 ,
	\end{equation}
	for any $z^1,z^2 \in B$. The inequality (\ref{stabili}) is often
	called \textit{stabilizability inequality}.
\end{definition}

\medskip
In what follows we state three theorems established in \cite{chueshov-yellow} that guarantee that quasi-stable systems guarantee the existence of a compact global attractor and enjoy properties such as finite dimension, smoothness, and existence of exponential attractors.

\begin{theorem} [Theorem 7.9.4, \cite{chueshov-yellow}] \label{Exist-attractor}\label{quasi-stable-result} 
Let Assumption \ref{assumption-quasi-stable} be in force.
Assume that the dynamical system $(X,S_t)$ is quasi-stable on every bounded forward invariant set $B$ in $X$. Then $(X,S_t)$ is asymptotically smooth.
\end{theorem}

\begin{theorem} [Theorem 7.9.6, \cite{chueshov-yellow}] \label{quasi-stable-result} Assume that the dynamical system $(X,S_t)$ possess a compact global attractor $\mathfrak{A}$ and is quasi-stable on $\mathfrak{A}$. Then the attractor $\mathfrak{A}$  has a finite fractal dimension $dim^X_f\mathfrak{A}$.
\end{theorem}

\begin{theorem} [Theorem 7.9.8, \cite{chueshov-yellow}] \label{regularity-result} Assume that the dynamical system $(X,S_t)$ possess a compact global attractor $\mathfrak{A}$ and is quasi-stable on $\mathfrak{A}$. Moreover, we assume that \eqref{stabili} holds with the function $c(t)$ possessing the property $c_{\infty}=\sup_{t\in \mathbb{R}^+}c(t)<\infty$. Then any full trajectory $\{(u(t);u_t(t);\theta(t)):t\in \mathbb{R}\}$ that belongs to the global attractor enjoys the following regularity properties,
	$$u_t\in L^{\infty}(\mathbb{R};Y)\cap C(\mathbb{R};Z),\quad u_{tt}\in L^{\infty}(\mathbb{R};Z),\quad \theta\in L^{\infty}(\mathbb{R};W).$$
	Moreover, there exists $R>0$ such that
	$$|u_t(t)|_Y^2+|u_{tt}(t)|_Z^2+|\theta_t(t)|_W^2\le R^2,\quad t\in \mathbb{R},$$
	where $R$ depends on the constant $c_{\infty}$, on the seminorm $n_Y$ in Definition \ref{Def-quasi-stable}, also on the embedding properties of $Y$ into $Z$.
\end{theorem}

\section{Main results: Theorem \ref{theo-main}}
Our main results on compact global attractors for the dynamical system $(\mathcal{H},S_t)$ are established in the theorem below.
\begin{theorem}
	\label{theo-main} Under the assumptions of Theorem \ref{theo-existence}, we have:
	\begin{itemize}
		\item[{\bf I}.] {\bf (Global attractor)} the dynamical system $(\mathcal{H},S_t)$ possesses a unique compact global attractor $\mathfrak{A}\subset \mathcal{H}$, which is characterized by the unstable manifold $\mathfrak{A}=\mathcal{M}^z(\mathcal{N})$, emanating from the set $\mathcal{N}$ of stationary solutions. Moreover, $\mathfrak{A}$ consists of full trajectories $\{z(t):t\in \mathbb{R}\}$ such that
		$$\lim_{t\rightarrow -\infty}\mbox{dist}_{\mathcal{H}}\left(z(t),\mathcal{N}\right)=\lim_{t\rightarrow +\infty}\mbox{dist}_{\mathcal{H}}\left(z(t),\mathcal{N}\right)=0.$$
	\end{itemize}
	In addition, if $p<p^*$:
	\begin{itemize}
		\item[{\bf II}.] {\bf (Finite-dimensionality)} the compact global attractor $\mathfrak{A}\subset \mathcal{H}$ has finite fractal dimension $dim^f_{\mathcal{H}}\mathfrak{A}$.
		\item[{\bf III}.] {\bf (Regularity)} any full trajectory $\{z(t):t\in \mathbb{R}\}$ from the attractor $\mathfrak{A}$ possesses the properties
		\begin{align*}
			||z(t)||^2_{\mathcal{H}}\le R_1^2\quad\mbox{and}\quad
			\|\Delta u_t(t)\|^2+\|u_{tt}(t)\|^2+\|\eta^t_t\|^2_{L^2_g(\mathbb{R}^+;V)}\le R_2^2,
		\end{align*}
		for all $t\in \mathbb{R}$, and some positive constants $R_1$ and $R_2$.
	\end{itemize}
\end{theorem}
\begin{proof}
	{\bf I.} In Proposition \ref{absorbing-set} we showed that the dynamical system $(\mathcal{H},S_t)$ has an absorbing set $\mathcal{B}$, which guarantees that $(\mathcal{H},S_t)$ is dissipative. In Proposition \ref{Prop-grad-syst} we show that $(\mathcal{H},S_t)$ is gradient, and in Corollary \ref{Cor-quasi-stable}, we show that $(\mathcal{H},S_t)$ is quasi-stable. Therefore, the existence of a compact global attractor $\mathfrak{A}$ and its geometric characterization $\mathfrak{A}=\mathcal{M}^z(\mathcal{N})$ follows directly from the combination of Theorems \ref{Exist-attractor}, \ref{theo-chueshov}, and \ref{theo-chueshov-caracterization}.\\
	{\bf II.} In Corollary \ref{Cor-quasi-stable} we prove that in the subcritical case $p<p*$ system $(\mathcal{H},S_t)$ is quasi-stable. Therefore, it follows the direct application of the Theorem \ref{quasi-stable-result} that the global attractor $\mathfrak{A}$ has finite fractal dimension.\\
	{\bf III.} The proof follows as a direct application of Theorem \ref{regularity-result} because the system is quasi-stable by Corollary \ref{Cor-quasi-stable}.
\end{proof}

\section{Complementation of the proof of Theorem \ref{theo-main}}
\subsection{Existence of bounded absorbing sets}
\begin{proposition}{\bf (Absorbing set)} \label{absorbing-set}
	Assume that Assumption \ref{A1} holds true, then the semigroup defined by problem \eqref{P}-\eqref{initial-condition0} has a bounded absorbing set in $\mathcal{H}$, i.e., for any bounded set $B\subset \mathcal{H}$ there exist $t_0(B)>0$ and $R>0$, such that
	\begin{align*}
		\|S_tz_0\|_{\mathcal{H}}\leq R,\quad \text{for all}\ t\geq t_0(B),
	\end{align*}
	where $z_0=(u_0,u_1,\eta_0)$ and $\mathcal{B}=\overline{B}_{\mathcal{H}}(0,R)$ is a bounded absorbing set of the dynamical system $(\mathcal{H},S_t).$
\end{proposition}
\begin{proof}
	Let us begin by fixing an arbitrary bounded set $B\subset \mathcal{H}$ and consider the solutions of problem \eqref{P}-\eqref{initial-condition0} given by
	$z(t)=(u(t),u_t(t),\eta^t) = S_tz_0$ with $z_0=(u_0,u_1,\eta_0)\in B.$
	In order, by arguments analogous to those used in the first estimate, we get
	\begin{align*}
		||z(t)||^2_{\mathcal{H}}\le  \frac{2(C_fmeas(\Gamma_1)+E(z_0))}{\varrho},\quad \forall t\ge 0.
	\end{align*}
	Thus, since $E(z_0)\le c_0(B)$, we have
	\begin{align}\label{est-z}
		||z(t)||^2_{\mathcal{H}}\le  \frac{2(C_fmeas(\Gamma_1)+c_B)}{\varrho}=:c_1(B),\quad \forall t\ge 0.
	\end{align}
	Next, let us define the modified energy functional
	\begin{align*}
		\widetilde{E}(z(t)):=E(z(t))+C_fmeas(\Gamma_1).
	\end{align*}
	\paragraph{Affirmation: $\widetilde{E}(z(t))$ dominates $||z(t)||_{\mathcal{H}}$.} Indeed, from Assumption \eqref{diff2} and embedding $V\hookrightarrow L^2(\Gamma_1)$, we have
	\begin{align}\label{first-x}
		\int_{\Gamma_1}F(u)d\Gamma\ge  -C_fmeas(\Gamma_1)-\frac{c_f\lambda_1}{2}\|\Delta u(t)\|^2.
	\end{align}
	Then, using that $\varrho:=1-\frac{c_f\lambda_1}{\kappa}>0$, it follows from \eqref{first-x} that
	\begin{eqnarray}\label{s'}
		E(z(t))&\geq&\frac{1}{2}\|u_t(t)\|^2+\frac{\varrho\kappa}{2}\|\Delta u(t)\|^2+\frac{1}{2}\|\eta^t\|^2_{L^2_{g}(\mathbb{R}^+;V)}
		-C_f meas(\Gamma_1)\nonumber\\
		&\ge& \frac{\varrho}{2}||z(t)||^2_{\mathcal{H}}-C_fmeas(\Gamma_1).
	\end{eqnarray}
	Therefore, from \eqref{s'}, we get
	\begin{align}\label{s}\widetilde{E}(z(t))\ge \frac{\varrho}{2}||z(t)||^2_{\mathcal{H}}.\end{align}
	Which proves the affirmation.
	
	\medskip
	Now, let’s define the functional
	\begin{align}\label{def_L1}
	E^{\delta}(z(t)):=\widetilde{E}(z(t))+\delta \phi_1(t)+\frac{2\delta}{1-\kappa}\phi_2(t),
	\end{align}
	where
	\begin{align*}
		\phi_1(t):=\int_{\Omega}u(t)u_t(t)dx\quad \mbox{and}\quad
		\phi_2(t):=-\int_{\Omega}u_t(t)\int_{0}^{\infty}g(s)\eta^t(s)dsdx.
	\end{align*}
	Taking the derivative of \eqref{def_L1} with respect to $t$, it holds that
	\begin{align}\label{def_L}
		\frac{d}{dt}E^{\delta}(z(t)):=\frac{d}{dt}\widetilde{E}(z(t))+\delta\phi'_1(t)+\frac{2\delta}{1-\kappa}\phi'_2(t).
	\end{align}
	\subsubsection*{Step 1.: estimate of term $\frac{d}{dt}\widetilde{E}(z(t))$}
	Multiplying the Eq. $\eqref{P}_1$ by $u_t$ and Eq. $\eqref{P}_2$ by $\eta^t$ and integrating over $\Omega$, we obtain that
	\begin{align}\label{energy_1}
		\frac{1}{2}\frac{d}{dt}\|u_t\|^2+(\eta^{t}_{t},\eta^t)_{g,V}=-\kappa(\Delta^2u,u_t)-(\Delta^2\eta^t,u_t)_{g}
		-(\eta^{t}_{s},\eta^t)_{g,V}+(u_t,\eta^t)_{g,V}.
	\end{align}
	
	We now deal with the above equality term by term. For the second term on the left-hand side of \eqref{energy_1}, integrating by part, we have
	\begin{align*}
		(\eta^{t}_{t},\eta^t)_{g,V}=\int_{0}^{\infty}g(\tau)\frac{1}{2}\frac{d}{dt}\|\Delta\eta^t(\tau)\|^2d\tau=\frac{1}{2}\frac{d}{dt}\|\eta^t\|^2_{L^2_{g}(\mathbb{R}^+;V)}.
	\end{align*}
	And similarly, for the third term on the right-hand side of \eqref{energy_1}, we find
	\begin{align*}
		-(\eta^{t}_{s},\eta^t)_{g,V}=\frac{1}{2}\int_{0}^{\infty}g'(\tau)\|\Delta\eta^t(\tau)\|^2d\tau.
	\end{align*}
	To deal with the first two terms on the right-hand side of \eqref{energy_1}, we use the Green-formula \cite[P68, Ch4, (3.20)]{Lag} to get, respectively,
	\begin{align*}
		\kappa(\Delta^2u,u_t)&=\kappa(\Delta u,\Delta u_t)+ \kappa\int_{\Gamma}\left[\frac{\partial(\Delta u)}{\partial\nu}u_t-\Delta u\frac{\partial u_t}{\partial\nu}\right]d\Gamma\\
		&=\frac{\kappa}{2}\frac{d}{dt}\|\Delta u\|^2+\kappa\int_{\Gamma}\left[\frac{\partial(\Delta u)}{\partial\nu}u_t-\Delta u\frac{\partial u_t}{\partial\nu}\right]d\Gamma,
	\end{align*}
	
	and
	
	\begin{align*}
		(\Delta^2\eta^t,u_t)_{g}=( \eta^t,u_t)_{g,V}+\int_{\Gamma}\left[\int_{0}^{\infty}g(\tau)\frac{\partial(\Delta \eta^t(\tau))}{\partial\nu}d\tau u_t-\int_{0}^{\infty}g(\tau)\Delta \eta^t(\tau)d\tau\frac{\partial u_t}{\partial\nu}\right]d\Gamma.
	\end{align*}
	Then, putting the above results into \eqref{energy_1} and taking into account boundary conditions \eqref{boundary-condition}, we obtain that
	\begin{align}\label{3}
		\frac{d}{dt}E(z(t))=\frac{1}{2}\int_{0}^{\infty}g'(s)\|\Delta\eta^t(s)\|^2ds-a(t)\|u_t\|^2_{\Gamma_1}.
	\end{align}
	Thus, from \eqref{hyp_M} and \eqref{kernel_g'}, and using that $\frac{d}{dt}\widetilde{E}(z(t))=\frac{d}{dt}E(z(t))$, it follows from \eqref{3} that
	\begin{align}\label{4}
		\frac{d}{dt}\widetilde{E}(z(t))\le -\frac{\alpha_2}{2}\|\eta^t\|^2_{L^2_g(\mathbb{R}^+;V)}-m_0\|u_t(t)\|^2_{\Gamma_1}.
	\end{align}
	\subsubsection*{Step 2.: estimate of term $\delta\phi'_1(t)$.}
	By Eq. \eqref{P} and the Green-formula, we have
	\begin{eqnarray}\label{step2-1}
		\phi'_1(t)&=&\int_{\Omega}u_t u_tdx+\int_{\Omega}uu_{tt}dx\nonumber\\
		&=&\|u_t(t)\|^2-\int_{\Omega}u\left[\kappa\Delta^2 u+\int_{0}^{\infty}g(\tau)\Delta^2\eta^t(\tau)d\tau\right]dx\nonumber\\
		&=&\|u_t(t)\|^2-\kappa\|\Delta u(t)\|^2-\underbrace{\int_{\Gamma_1}f(u)ud\Gamma}_{I}-a(t)\int_{\Gamma_1}u_tud\Gamma\nonumber\\
		&&-\int_{0}^{\infty}g(\tau)\int_{\Omega}\Delta\eta^t(\tau)\Delta u(t)dxd\tau.
	\end{eqnarray}
	From Assumption \eqref{diff2} and embedding $V\hookrightarrow L^2(\Gamma_1)$, we have
	\begin{align*}
		-I=-\int_{\Gamma_1}f(u)ud\Gamma\le -\int_{\Gamma_1}F(u)d\Gamma+\frac{c_f}{2}\|u(t)\|^2_{\Gamma_1}\le -\int_{\Gamma_1}F(u)d\Gamma+\frac{c_f\lambda_1}{2}\|\Delta u(t)\|^2.
	\end{align*}
	Then, returning to \eqref{step2-1} and using that $\varrho:=1-\frac{c_f\lambda_1}{\kappa}$, we get
	\begin{eqnarray}\label{step2-2}
		\delta\phi'_1(t)&\le&
		\delta\|u_t(t)\|^2-\frac{\delta\kappa}{2}\|\Delta u(t)\|^2-\frac{\delta\varrho\kappa}{2}\|\Delta u(t)\|^2-\delta\int_{\Gamma_1}F(u)d\Gamma\nonumber\\
		&&-\underbrace{\delta a(t)\int_{\Gamma_1}u_tud\Gamma}_{I_1}-\underbrace{\delta\int_{0}^{\infty}g(\tau)\int_{\Omega}\Delta\eta^t(\tau)\Delta u(t)dxd\tau}_{I_2}.
	\end{eqnarray}
	For one thing, Theorem \ref{theo-existence} means that $\|u(t)\|_{\Gamma_1}^2\leq c_2(B)$ for all $t\geq0$. For another, $M$ is $C^1$, which implies $M$ is uniformly bounded on $[0,c_2(B)]$, i.e., there exists $c_3(B)>0$ such that $c_3(B)=\max\{M(s);s\in [0,c_2(B)]\}$, together with the fact that
	$\|u\|^2_{\Gamma_1}\leq\lambda_1\|\Delta u\|^2$, then
	\begin{eqnarray*}
		-I_1=-\delta a(t)\int_{\Gamma_1}u_tud\Gamma&\le &\delta M(\|u(t)\|^2_{\Gamma_1})\|u_t(t)\|_{\Gamma_1}\|u(t)\|_{\Gamma_1}\\
		&\le &\delta M(\|u(t)\|^2_{\Gamma_1})\lambda_1^{1/2}\|u_t(t)\|_{\Gamma_1}\|\Delta u(t)\|_{\Gamma_1}\\
		&\le&\frac{\delta\varrho\kappa}{8}\|\Delta u(t)\|^2+\frac{2\delta\lambda_1c_3(B)^2}{\varrho\kappa}\|u_t(t)\|^2_{\Gamma_1}.
	\end{eqnarray*}
	From Young's inequality, we find
	\begin{align*}
		-I_2=-\delta\int_{0}^{\infty}g(\tau)\int_{\Omega}\Delta\eta^t(\tau)\Delta u(t)dxd\tau\leq \frac{\delta\varrho\kappa}{8}\|\Delta u(t)\|^2+\frac{2\delta(1-\kappa)}{\varrho\kappa}\|\eta^t\|^2_{L^{2}_{g}(\mathbb{R}^+;V)}.
	\end{align*}
	Using $I_1$ and $I_2$, it follows from \eqref{step2-2} that
	\begin{eqnarray}\label{step2-3}
		\delta\phi'_1(t)&\le&\delta\|u_t(t)\|^2-\frac{\delta\kappa}{2}\|\Delta u(t)\|^2-\frac{\delta\varrho\kappa}{4}\|\Delta u(t)\|^2-\delta\int_{\Gamma_1}F(u)d\Gamma_1\nonumber\\
		&&+\frac{2\delta\lambda_1c_3(B)^2}{\varrho\kappa}\|u_t(t)\|^2_{\Gamma_1}+\frac{2\delta(1-\kappa)}{\varrho\kappa}\|\eta^t\|^2_{L^{2}_{g}(\mathbb{R}^+;V)}.
	\end{eqnarray}
	\subsubsection*{Step 3.: estimate of term $\frac{2\delta}{1-\kappa}\phi'_2(t)$.}
	Using Eq. \eqref{P} as well as the Green-formula, we obtain
	\begin{eqnarray}\label{step3-1}
		\phi'_2(t)&=&-\int_{\Omega}u_{tt}\int_{0}^{\infty}g(\tau)\eta^t(\tau)d\tau dx-
		\int_{\Omega}u_{t}\int_{0}^{\infty}g(\tau)\eta^t_{t}(\tau)d\tau dx\nonumber\\
		&=&\int_{\Omega}\left[\kappa\Delta^2 u+\int_{0}^{\infty}g(\tau)\Delta^2\eta^t(\tau)d\tau\right]\int_{0}^{\infty}g(\tau)\eta^t(\tau)d\tau dx\nonumber\\
		&&\quad-\int_{\Omega}u_{t}\int_{0}^{\infty}g(\tau)\eta^t_{t}(\tau)d\tau dx\nonumber\\
		&=&-\underbrace{\int_{0}^{\infty}g(\tau)d\tau}_{=1-\kappa}\|u_t(t)\|^2+\kappa\int_{\Omega}\int_{0}^{\infty}g(\tau)\Delta\eta^t(\tau)\Delta u(t)dxd\tau
		\nonumber\\
		&&\quad+\int_{\Omega}\left[\int_{0}^{\infty}g(\tau)\Delta\eta^t(\tau)d\tau\right]^2dx+\int_{\Gamma_1}f(u)\int_{0}^{\infty}g(\tau)\eta^t(\tau)d\tau d\Gamma\nonumber\\
		&&\quad+a(t)\int_{\Gamma_1}u_t \int_{0}^{\infty}g(\tau)\eta^t(\tau)d\Gamma-\int_{0}^{\infty}g'(\tau)(\eta^t(\tau),u_t)d\tau.
	\end{eqnarray}
	Thus, from \eqref{step3-1}, we have
	\begin{align}
		\label{step3-2} 	\frac{2\delta}{(1-\kappa)}\phi'_2(t)\le-2\delta \|u_t(t)\|^2+\sum_{j=1}^{5}J_j.
	\end{align}
	We now estimate $J_1,\cdots, J_5$ one by one. First of all, by Young's inequality,
	\begin{eqnarray*}
		J_1&=&\frac{2\delta\kappa}{(1-\kappa)}\int_{0}^{\infty}g(\tau)\int_{\Omega}\Delta\eta^t(\tau)\Delta u(t)dxd\tau\\
		&\le&\frac{2\delta\kappa}{(1-\kappa)}\int_{0}^{\infty}g(\tau)\|\Delta\eta^t(\tau)\|\|\Delta u(t)\|d\tau\\
		&\le&\frac{2\delta\kappa}{(1-\kappa)^{1/2}}\|\eta^t\|_{L^2_g(\mathbb{R}^+;V)}\|\Delta u(t)\|\\
		&\le& \frac{8\delta\kappa}{\varrho(1-\kappa)}\|\eta^t\|^2_{L^2_g(\mathbb{R}^+;V)}+\frac{\delta\varrho\kappa}{8}\|\Delta u(t)\|^2.
	\end{eqnarray*}
	Applying H\"older inequality, we get
	\begin{align*}
		J_2=\frac{2\delta}{(1-\kappa)}\int_{\Omega}\left[\int_{0}^{\infty}g(\tau)\Delta\eta^t(\tau)d\tau\right]^2dx
		\le 2\delta \|\eta^t\|^2_{L^2_{g}(\mathbb{R}^+;V)}.
	\end{align*}
	From Assumption \eqref{diff}, H\"older inequality with $\frac{p}{2(p+1)}+\frac{1}{2(p+1)}+\frac{1}{2}=1$, embedding $V\hookrightarrow L^{2(p+1)}(\Gamma_1)$, and Young's inequality, we have
	\begin{eqnarray*}
		J_3&=&\frac{2\delta}{(1-\kappa)}\int_{0}^{\infty}g(\tau)\int_{\Gamma_1}f(u)\eta^t(\tau)d\Gamma d\tau\\
		&\le&\frac{2\delta C_{f'}}{(1-\kappa)}\int_{0}^{\infty}g(\tau)\int_{\Gamma_1}\left[1+|u|^p\right]|u||\eta^t(\tau)|d\Gamma d\tau\\
		&\le&\frac{2\delta C_{f'}c_4(B)}{(1-\kappa)}\int_{0}^{\infty}g(\tau)\|\Delta \eta^t(\tau)\| d\tau\|\Delta u(t)\|\\
		&\le&\frac{2\delta C_{f'}c_4(B)}{(1-\kappa)^{1/2}}\|\eta^t\|_{L^2_g(\mathbb{R}^+;V)}\|\Delta u(t)\|\\
		&\le&\frac{8\delta C^2_{f'}c_4(B)^2}{\varrho\kappa(1-\kappa)}\|\eta^t\|^2_{L^2_g(\mathbb{R}^+;V)}+\frac{\delta\varrho\kappa}{8}\|\Delta u(t)\|^2.
	\end{eqnarray*}
	Using that $M\in C^1(\mathbb{R}^+)$, \eqref{est-z}, Young's inequality, and $\|\eta^t\|^2_{\Gamma_1}\leq\lambda_3\|\Delta \eta^t\|^2$,
	\begin{eqnarray*}
		J_4&=&\frac{2\delta}{(1-\kappa)}\int_{\Gamma_1}a(t)u_t\int_{0}^{\infty}g(\tau)\eta^t(\tau)d\tau d\Gamma\\
		&\le &\frac{2\delta c_3(B)}{(1-\kappa)}\|u_t(t)\|_{\Gamma_1}\int_{0}^{\infty}g(\tau)\|\eta^t(\tau)\|_{\Gamma_1}d\tau\\
		&\le &\frac{2\delta c_3(B)\lambda_3^{1/2}}{(1-\kappa)}\|u_t(t)\|_{\Gamma_1}\int_{0}^{\infty}g(\tau)\|\Delta \eta^t(\tau)\|d\tau\\
		&\le &\frac{2\delta c_3(B)\lambda_3^{1/2}}{(1-\kappa)^{1/2}}\|u_t(t)\|_{\Gamma_1}\|\eta^t\|_{L^2_g(\mathbb{R}^+,V)}\\
		&\le & \frac{\delta c_3(B)\lambda_3^{1/2}}{(1-\kappa)^{1/2}}\left[\|u_t(t)\|^2_{\Gamma_1}+\|\eta^t\|^2_{L^2_g(\mathbb{R}^+;V)}\right].
	\end{eqnarray*}
	Finally, from Assumption \eqref{kernel_g'}, embedding $V\hookrightarrow L^{2}(\Omega)$, and Young's inequality, we arrive at
	\begin{eqnarray*}
		J_5&=&-\frac{2\delta}{(1-\kappa)}\int_{0}^{\infty}g'(\tau)(\eta^t(\tau),u_t)d\tau\\
		&\le& \frac{2\delta \alpha_1}{(1-\kappa)}\int_0^{\infty}g(\tau)\|\eta^t(\tau)\|\|u_t(t)\|d\tau\\
		&\le& \frac{2\delta \alpha_1\lambda_2^{1/2}}{(1-\kappa)^{1/2}}\|\eta^t\|_{L^2_g(\mathbb{R}^+;V)}\|u_t(t)\|\\
		&\le& \frac{2\delta\alpha_1^2\lambda_2}{1-\kappa}\|\eta^t\|^2_{L^2_g(\mathbb{R}^+;V)}+\frac{\delta}{2}\|u_t(t)\|^2.
	\end{eqnarray*}
	Thus, from the estimates of terms $J_1,\cdots,J_5$, it follows from \eqref{step3-2} that
	\begin{eqnarray}
		\label{step3-3} \frac{2\delta}{(1-\kappa)}\phi'_2(t)&\le&-\frac{3\delta}{2} \|u_t(t)\|^2+\frac{\delta\varrho\kappa}{4}\|\Delta u(t)\|^2+\frac{\delta c_3(B)\lambda_3^{1/2}}{(1-\kappa)^{1/2}}\|u_t(t)\|^2_{\Gamma_1}\nonumber\\
		&&+\delta c_5(B)\|\eta^t\|^2_{L^2_g(\mathbb{R}^+;V)},
	\end{eqnarray}
	where $$c_5(B):=\frac{8\kappa}{\varrho(1-\kappa)}+2+\frac{8C^2_{f'}c_4(B)^2}{\varrho\kappa(1-\kappa)}+\frac{c_3(B)\lambda_3^{1/2}}{(1-\kappa)^{1/2}}+ \frac{2\alpha_1^2\lambda_2}{1-\kappa}.$$
\subsubsection*{Completion of the proof}
	Thus, using estimates \eqref{4}, \eqref{step2-3}, and \eqref{step3-3}, it follows from \eqref{def_L} that
	\begin{eqnarray}\label{step4-1}
		\frac{d}{dt}E^{\delta}(z(t))&\le&-\frac{\delta}{2}\|u_t(t)\|^2-\frac{\delta\kappa}{2}\|\Delta u(t)\|^2-\left[\frac{\alpha_2}{2}-\delta \left(\frac{2(1-\kappa)}{\varrho\kappa}+c_5(B)\right)\right]\|\eta^t\|^2_{L^2_g(\mathbb{R}^+;V)}\nonumber\\
		&&-\delta\int_{\Gamma_1}F(u)d\Gamma-\left[m_0-\delta\left(\frac{2\lambda_1c_3(B)^2}{\varrho\kappa}+\frac{c_3(B)\lambda_3^{1/2}}{(1-\kappa)^{1/2}}\right)\right]\|u_t(t)\|^2_{\Gamma_1}.
	\end{eqnarray}
	Now, choosing a $\delta>0$ conveniently small such that
	$$\frac{\alpha_2}{2}-\delta \left(\frac{2(1-\kappa)}{\varrho\kappa}+c_5(B)\right)\ge \frac{\delta}{2}\quad \mbox{and}\quad m_0-\delta\left(\frac{2\lambda_1c_3(B)^2}{\varrho\kappa}+\frac{c_3(B)\lambda_3^{1/2}}{(1-\kappa)^{1/2}}\right)\ge 0,$$
	from \eqref{step4-1}, we obtain that
	\begin{eqnarray}\label{step4-2}
		\frac{d}{dt}E^{\delta}(z(t))&\le&-\delta E(z(t)).
	\end{eqnarray}
	Using the definition of $\widetilde{E}$ we obtain from \eqref{step4-2} that
	\begin{eqnarray}\label{step4-3}
		\frac{d}{dt}E^{\delta}(z(t))&\le&-\delta \widetilde{E}(z(t))+\delta C_fmeas(\Gamma_1).
	\end{eqnarray}
	\paragraph{Affirmation: $E^{\delta}(z(t))\sim \widetilde{E}(z(t))$} Indeed,
	\begin{eqnarray*}
		\left|E^{\delta}(z(t))-\widetilde{E}(z(t))\right|\le \delta|\phi_1'(t)|+\frac{2\delta}{(1-\kappa)}|\phi'_2(t)|.
	\end{eqnarray*}
	From embedding $V\hookrightarrow L^2(\Omega)$ and \eqref{s'}, we have
	\begin{align*}
		|\phi_1(t)|\le \lambda_0^{1/2}\|u_t(t)\|\|\Delta u(t)\|\le \frac{2\lambda_0^{1/2}}{\varrho\kappa}\widetilde{E}(z(t)),
	\end{align*}
	and, similarly,
	\begin{eqnarray*}
		|\phi_2(t)|&\le&\lambda_2^{1/2}\|u_t(t)\|\int_{0}^{\infty}g(\tau)\|\Delta\eta^t(\tau)\|d\tau\\
		&\le&\lambda_2^{1/2}(1-\kappa)^{1/2}\|u_t(t)\|\|\eta^t\|_{L^2_g(\mathbb{R}^+;V)}\le\frac{ 2\lambda_2^{1/2}(1-\kappa)^{1/2}}{\rho}\widetilde{E}(z(t)).
	\end{eqnarray*}
	Considering the last two estimates we obtain that
	\begin{eqnarray*}
		\left|E^{\delta}(z(t))-\widetilde{E}(z(t))\right|\le \delta C_0\widetilde{E}(z(t)),
	\end{eqnarray*}
	where $$C_0:=\frac{2\lambda_0^{1/2}}{\varrho\kappa}+\frac{ 2\lambda_2^{1/2}(1-\kappa)^{1/2}}{\rho}.$$
	Still choosing a small $\delta$ such that $\delta\le \frac{1}{2C_0}$, we obtain that
	\begin{align}\label{equivE}
		\frac{1}{2}\widetilde{E}(z(t))\le E^{\delta}(z(t))\le \frac{3}{2}\widetilde{E}(z(t)).
	\end{align}
	Which proves that $E^{\delta}\sim \widetilde{E}$.
	
	Hence, using \eqref{equivE}, from \eqref{step4-3}, we have
	\begin{eqnarray}\label{step4-4}
		\frac{d}{dt}E^{\delta}(z(t))+\frac{2\delta}{3}E^{\delta}(z(t))\le \delta C_fmeas(\Gamma_1).
	\end{eqnarray}
	Applying Gronwall's lemma in \eqref{step4-4}, one gets
	\begin{eqnarray*}
		E^{\delta}(z(t))\le e^{-\frac{2\delta}{3}t}E^{\delta}(z(0))+\frac{3}{2}C_fmeas(\Gamma_1).
	\end{eqnarray*}
	Again using \eqref{equivE}, we get
	\begin{eqnarray*}
		\widetilde{E}(z(t))\le 2e^{-\frac{2\delta}{3}t}\widetilde{E}(z(0))+3C_fmeas(\Gamma_1).
	\end{eqnarray*}
	Since $\widetilde{E}(z(0)\le c_6(B)$, from \eqref{s}, we obtain that
	\begin{eqnarray*}
		||z(t)||^2_{\mathcal{H}}\le \frac{4}{\varrho}e^{-\frac{2\delta}{3}t}\widetilde{E}(z(0))+\frac{6C_fmeas(\Gamma_1)}{\varrho}.
	\end{eqnarray*}
	and
	$$\lim\sup_{t\rightarrow \infty}||z(t)||^2_{\mathcal{H}}\le \frac{6C_fmeas(\Gamma_1)}{\varrho}=:R^2.$$
	This implies that the closed ball $\mathcal{B}=\overline{B}_{\mathcal{H}}(0,R)$ is a bounded absorbing sets of dynamical system $(\mathcal{H},S_t)$. The proof is complete.
\end{proof}

\subsection{Gradient system}
\begin{proposition} {\bf (Gradient system)}\label{Prop-grad-syst}
	Under the assumptions of Theorem \ref{theo-main}, the energy functional  $E:\mathcal{H}\to \mathbb{R}$ associated with problem \eqref{P}-\eqref{initial-condition0}
	$$E(S_tz):=\frac{1}{2}||S_tz||^2_{\mathcal{H}}+\int_{\Gamma_1}F(u)d\Gamma,\quad z=(u,u_t,\eta^t),$$
	is a strict Lyapunov functional for the dynamical system $(\mathcal{H},S_t)$. Which means, $(\mathcal{H},S_t)$ is a gradient dynamical system.
\end{proposition}
\begin{proof}
	We consider $z\in \mathcal{H}$ and set $(u (t),u_t(t),\eta^t)=S_tz$, for all $t\ge 0$. From \eqref{A1} and \eqref{kernel_g'}, we have
	\begin{equation}\label{grad-1}
		\frac{d}{dt}E(S_tz)=\frac{1}{2}\int_{0}^{\infty}g'(s)\|\Delta\eta^t(s)\|^2ds-a(t)\|u_t\|^2_{\Gamma_1}\leq0,
	\end{equation}
	which means that the function $t\mapsto E(S_tz)$ is a non-increasing function, for any $z\in \mathcal{H}$. Now, we suppose that for some initial data $z_0=(u_0,u_1,\eta_0)\in \mathcal{H}$ one has
	\begin{align*}
		E(S_tz_0)= E(z_0),\quad \forall   t> 0.
	\end{align*}
	Then, it follows from \eqref{grad-1} that
	\begin{align}
		-\frac{1}{2}\int_{0}^{\infty}g'(s)\|\Delta\eta^t(s)\|^2ds+a(t)\|u_t\|^2_{\Gamma_1}=0,\quad \forall t\ge 0.
	\end{align}
	Using that both terms have same sign, from \eqref{kernel_g'}, we conclude that
	\begin{align}
		0=-\frac{1}{2}\int_{0}^{\infty}g'(s)\|\Delta\eta^t(s)\|^2ds\ge \frac{\alpha_2}{2}\|\eta^t\|^2_{L^2_g(\mathbb{R}^+;V)}\ge 0.
	\end{align}
	This implies that $\eta^t(x,s)=0$ a.e. in $\Omega$, $t,s\ge 0$. Hence, from $\eqref{P}_2$, we have $u_t(x,t)=0$ a.e. in $\Omega$, $t\ge 0$. Therefore, $S_t(u_0,u_1,\eta_0)=(u_0,u_1,\eta_0)$ is a fix point of $S_t$. Consequently, $E(z)$ is a strict Lyapunov functional.
\end{proof}
\begin{remark}
	Readers are referred to \cite[Definition 2.4.1, P80]{chueshov-yellow} for more details about gradient systems.
\end{remark}
\subsection{Stabilizability estimate}
Now, we are in the position to investigate the stabilizability of dynamical systems generated by \eqref{P}-\eqref{initial-condition0}.
\begin{proposition}\label{prop-stab-ineq}
	Let us assume that the hypotheses of Theorem \ref{theo-existence} are valid. Given a bounded set $B=\{U\in \mathcal{H}|\,||U||_{\mathcal{H}}\le R\}\subset \mathcal{H}$, let $S_tz^1_0=z^1(t) =
	(u(t),u_t(t),\eta^t)$ and $S_tz^2_0=z^2(t)=(v(t),v_t(t),\xi^t)$ be two weak solutions of problem \eqref{P}-\eqref{initial-condition0} such that $z^1(0)=z^1_0=(u_0,u_1,\eta_0)$
	and $z^2(0) = z^2_0=(v_0, v_1,\xi_0)$ are in $B$.  Then, there exist positive constants $\epsilon, C_{R,\epsilon}$ such that
	\begin{align}\label{inequality-stab}	||S_tz^1_0-S_tz^2_0||^2_{\mathcal{H}}\le 3e^{-\frac{\epsilon}{3}t}||z^1_0-z^2_0||^2_{\mathcal{H}}+C_{R,\epsilon}\int_0^te^{-\frac{\epsilon}{3}(t-s)}\|u(s)-v(s)\|^2_{2(p+1),\Gamma_1}ds.
	\end{align}
\end{proposition}
\begin{proof}
	Let $w = u-v$, $\zeta^t=\eta^t-\xi^t$. Then $z^1(t)-z^2(t)=(w(t),w_t(t),\zeta^t)$
	satisfies the problem \eqref{PD}-\eqref{initial-condition11}. In order, we consider the functional
	$$G(t)=\frac{1}{2}||S_tz^1_0-S_tz^2_0||^2_{\mathcal{H}}+\epsilon \varphi_1(t)+\frac{2\epsilon}{1-\kappa}\varphi_2(t),$$
	where $\epsilon>0$, and $$\varphi_1(t):=\int_{\Omega}w_t wdx\quad\mbox{and}\quad \varphi_2(t):=-\int_{\Omega}w_t(t)\int_0^{\infty}g(\tau)\zeta^t(\tau)d\tau.$$
	\paragraph{ Affirmation $G(t)\approx  \frac{1}{2}||S_tz^1_0-S_tz^2_0||^2_{\mathcal{H}}$.} Indeed, from embedding $V\hookrightarrow L^2(\Omega)$, we have
	\begin{eqnarray*}
		\left|G(t)-\frac{1}{2}||S_tz^1_0-S_tz^2_0||^2_{\mathcal{H}}\right|&\le& \epsilon\|w_t(t)\|\|w(t)\|+\frac{2\epsilon}{1-\kappa}\int_{0}^{\infty}g(\tau)\|\zeta^t(\tau)\|\|w_t(t)\|d\tau\\
		&\le&\epsilon\lambda_0^{1/2}\|w_t(t)\|\|\Delta w(t)\|+\frac{2\epsilon\lambda_2^{1/2}}{1-\kappa}\int_{0}^{\infty}g(\tau)\|\Delta\zeta^t(\tau)\|\|w_t(t)\|d\tau\\
		&\le&\epsilon\lambda_0^{1/2}\|w_t(t)\|\|\Delta w(t)\|+\epsilon\lambda_2^{1/2}\|w_t(t)\|^2+\frac{\epsilon\lambda_2^{1/2}}{1-\kappa}\|\zeta^t\|^2_{L^2_g(\mathbb{R}^+;V)}\\
		&\le&\epsilon\max\{\lambda_0^{1/2},\lambda_2^{1/2}\}\left(1+\frac{1}{4\kappa}+\frac{1}{1-\kappa}\right)||S_tz^1_0-S_tz^2_0||^2_{\mathcal{H}}.
	\end{eqnarray*}
	Hence, from the above inequality, we can choose $\epsilon>0$ appropriate such that
	\begin{align}\label{equiv} \frac{1}{2}||S_tz^1_0-S_tz^2_0||^2_{\mathcal{H}}\le G(t)\le \frac{3}{2}||S_tz^1_0-S_tz^2_0||^2_{\mathcal{H}}.
	\end{align}
	Taking the derivative of $G(t)$ with respect to $t$, we get
	\begin{eqnarray}
		\label{asymp-0}\frac{d}{dt}G(t)=\frac{1}{2}\frac{d}{dt}||S_tz^1_0-S_tz^2_0||^2_{\mathcal{H}}+\epsilon\varphi'_1(t)+\frac{2\epsilon}{1-\kappa}\varphi'_2(t).
	\end{eqnarray}
	In what follows we will estimate the terms on the right-hand side of \eqref{asymp-0}. This will be accomplished in three steps.
	\subsubsection*{Step 1.: estimate of term $\frac{1}{2}\frac{d}{dt}||S_tz^1_0-S_tz^2_0||^2_{\mathcal{H}}$.}
	First of all, multiplying equation $\eqref{P}_1$ by $w_t$, equation $\eqref{P}_2$ by $\zeta^t$, integrating over $\Omega$, and by arguments similar to those used in the proof of the Theorem \ref{theo-existence}-${\bf II.}$, we obtain
	\begin{align}	\label{asymp-2}
		\frac{1}{2}\frac{d}{dt}||S_tz^1_0-S_tz^2_0||^2_{\mathcal{H}}\le -\frac{\alpha_2}{2}\|\zeta^t\|^2_{L^2_g(\mathbb{R}^+;V)}-\frac{m_0}{2}\|w_t(t)\|^2_{\Gamma_1}+ C_R\left[\,\|w(t)\|^2_{\Gamma_1}+\|w(t)\|^2_{2(p+1),\Gamma_1}\,\right].
	\end{align}
	\subsubsection*{Step 2.: estimate of term $\epsilon\varphi'_1(t)$.}
	Next, differentiating $\varphi_1$ with respect to $t$, and substituting
	$$w_{tt}=-\kappa \Delta^2 w-\int_0^{\infty}g(\tau)\Delta^2\zeta^t(\tau)d\tau,$$
	in the obtained expression, it holds that
	\begin{align}\label{asymp-3}
		\varphi'_1(t)=\|w_t(t)\|^2-\kappa\|\Delta w(t)\|^2-\underbrace{\int_{\Gamma}(\frac{\partial(\Delta w)}{\partial \nu}w-\frac{\partial w}{\partial\nu}\Delta w)d\Gamma-\int_{\Omega}w\int_0^{\infty}g(\tau)\Delta^2\zeta^t(\tau)d\tau dx}_{J}.
	\end{align}
	From Green-formula and boundary conditions \eqref{boundary-condition2}, we can rewrite
	\begin{eqnarray*}
		J&=&\underbrace{-\int_{\Gamma_1}\left[\,f(u)-f(v)\,\right]wd\Gamma}_{J'}-\underbrace{\int_{\Gamma_1}\left[\,a(u,t)u_t-a(v,t)v_t\,\right]wd\Gamma}_{J''}\nonumber\\
		&&-\underbrace{\int_0^{\infty}g(\tau)\int_{\Omega}\Delta \zeta^t(\tau)\Delta w(t)dxd\tau}_{J'''}.
	\end{eqnarray*}
	From Assumption \eqref{diff}, H\"older inequality with $\frac{p}{p+2}+\frac{2}{p+2}=1$, and embedding $V\hookrightarrow L^{p+2}(\Gamma_1)$, we have
	\begin{eqnarray*}
		J'\le C_R\|w(t)\|^2_{p+2,\Gamma_1}.
	\end{eqnarray*}
	Using that $M\in C^1(\mathbb{R}^+)$, Mean Value Theorem, and Young's inequality, we have
	\begin{eqnarray*}
		J''&=&-M(\|u(t)\|^2_{\Gamma_1})\int_{\Gamma_1}w_twd\Gamma-\left[\,M(\|u(t)\|^2_{\Gamma_1})-M(\|v(t)\|^2_{\Gamma_1})\,\right]\int_{\Gamma_1}v_twd\Gamma\\
		&\le& C_R\|w_t(t)\|_{\Gamma_1}\|w(t)\|_{\Gamma_1}+C_R\|w(t)\|^2_{\Gamma_1}\\
		&\le&C_R\|w(t)\|^2_{\Gamma_1}+\frac{m_0}{4}\|w_t(t)\|^2_{\Gamma_1}.
	\end{eqnarray*}
	Finally, from Young's inequality
	\begin{eqnarray*}
		J'''&\le&\int_{0}^{\infty}g(\tau)\|\Delta \zeta^t(\tau)\|\|\Delta w(t)\|d\tau\\
		&\le &\int_0^{\infty}g(\tau)\left[\,\frac{(1-\kappa)}{\kappa}\|\Delta \zeta^t(\tau)\|^2+\frac{\kappa}{4(1-\kappa)}\|\Delta w(t)\|^2\,\right]d\tau\\
		&=&\frac{(1-\kappa)}{\kappa}\|\zeta^t\|^2_{L^{2}_g(\mathbb{R}^+;V)}+\frac{\kappa}{4}\|\Delta w(t)\|^2.
	\end{eqnarray*}
	Returning to \eqref{asymp-3} and using $L^{2(p+1)}(\Gamma_1)\hookrightarrow L^{p+2}(\Gamma_1)\hookrightarrow L^{2}(\Gamma_1)$, we get
	\begin{align}\label{asymp-4}
		\epsilon\varphi'_1(t)\le\;&\epsilon\|w_t(t)\|^2-\frac{3\kappa\epsilon}{4}\|\Delta w(t)\|^2+\frac{m_0\epsilon}{4}\|w_t(t)\|^2_{\Gamma_1}
		+\frac{(1-\kappa)\epsilon}{\kappa}\|\zeta^t\|^2_{L^2_g(\mathbb{R}^+;V)}\nonumber\\
		&+\epsilon C_R\|w(t)\|^2_{2(p+1),\Gamma_1}.
	\end{align}
	\subsubsection*{Step 3.: estimate of term $\frac{2\epsilon}{1-\kappa}\varphi'_2(t)$.}
	Now, taking the derivative of $\varphi_2(t)$ with respect to $t$ and again using
	$$w_{tt}=-\kappa \Delta^2 w-\int_0^{\infty}g(\tau)\Delta^2\zeta^t(\tau)d\tau$$ and boundary conditions \eqref{boundary-condition2}, we obtain
	\begin{eqnarray}\label{asymp-5}
		\frac{2\epsilon}{(1-\kappa)}\varphi_2'(t)=\sum_{j=1}^5J_j-2\epsilon\|w_t(t)\|^2,
	\end{eqnarray}
	where
	\begin{eqnarray*}
		J_1&=&\frac{2\epsilon\kappa}{1-\kappa}\int_0^{\infty}g(\tau)\int_{\Omega}\Delta \zeta^t(\tau)\Delta w(t)dxd\tau,\\
		J_2&=&	\frac{2\epsilon}{(1-\kappa)}\int_{\Omega}\left(\int_0^{\infty}g(\tau)\Delta \zeta^t(\tau)d\tau\right)^2dx,\\
		J_3&=&	\frac{2\epsilon}{(1-\kappa)}\int_{\Gamma_1}\left[\,f(u)-f(v)\,\right]\int_0^{\infty}g(\tau)\zeta^t(\tau)d\tau d\Gamma,\\
		J_4&=&	\frac{2\epsilon}{(1-\kappa)}\int_{\Gamma_1}\left[\,a_u(t)u_t-a_v(t)v_t\,\right]\int_0^{\infty}g(\tau)\zeta^t(\tau)d\tau d\Gamma,\\
		J_5&=&-	\frac{2\epsilon}{(1-\kappa)}\int_0^{\infty}g'(\tau)\int_{\Omega}\zeta^t(\tau)w_t(t)dxd\tau.
	\end{eqnarray*}
	Now let's estimate the terms $J_1,\cdots,J_5$. First, from Young's inequality, we have
	\begin{eqnarray*}
		J_1&\le& \frac{2\epsilon\kappa}{1-\kappa}\int_0^{\infty}g(\tau)\|\Delta \zeta^t(\tau)\|\|\Delta w(t)\|d\tau\\
		&\le& \frac{2\epsilon\kappa}{1-\kappa}\int_0^{\infty}g(\tau)\left[\,4\|\Delta \zeta^t(\tau)\|^2+\frac{1}{16}\|\Delta w(t)\|^2\,\right]d\tau\\
		&\le&\frac{8\epsilon\kappa}{1-\kappa}\|\zeta^t\|^2_{L^2_g(\mathbb{R}^+;V)}+\frac{\epsilon\kappa}{8}\|\Delta w(t)\|^2.
	\end{eqnarray*}
	From H\"older inequality, it's easy to see that
	\begin{eqnarray*}
		J_2&\le& \frac{2\epsilon}{(1-\kappa)}\int_{\Omega}\left(\int_0^{\infty}g(\tau)\Delta \zeta^t(\tau)d\tau\right)^2dx\le 2\epsilon\|\zeta^t\|^2_{L^2_g(\mathbb{R}^+;V)}.
	\end{eqnarray*}
	Using Assumption \eqref{diff}, H\"older inequality with $\frac{p}{2(p+1)}+\frac{1}{2(p+1)}+\frac{1}{2}=1$, and embedding $V\hookrightarrow L^2(\Gamma_1)$, we have
	\begin{eqnarray*}
		J_3&=&	\frac{2\epsilon C_{f'}}{(1-\kappa)}\int_0^{\infty}g(\tau)\int_{\Gamma_1}\left[\,1+(|u|+|v|)^p\,\right]|w||\zeta^t(\tau)|d\Gamma d\tau \\
		&\le&\frac{4\epsilon C_{f'}}{(1-\kappa)}\int_0^{\infty}g(\tau)\left(\int_{\Gamma_1}\left[\,1+(|u|+|v|)^{2(p+1)}\,\right]d\Gamma\right)^{\frac{p}{2(p+1)}}\|w(t)\|_{2(p+1),\Gamma_1}\|\zeta^t(\tau)\|_{\Gamma_1} d\tau\\
		&\le&\epsilon C_R\|w(t)\|^2_{2(p+1),\Gamma_1}+\epsilon\|\zeta^t\|^2_{L^2_g(\mathbb{R}^+;V)}.
	\end{eqnarray*}
	Using that $M\in C^1(\mathbb{R}^+)$, embedding $V\hookrightarrow L^{2}(\Gamma_1)$, and Young's inequality, we have
	\begin{eqnarray*}
		J_4&=&	\frac{2\epsilon}{(1-\kappa)}\int_{\Gamma_1}\left[\,a_u(t)u_t-a_v(t)v_t\,\right]\int_0^{\infty}g(\tau)\zeta^t(\tau)d\tau d\Gamma\\
		&=&	\frac{2\epsilon}{(1-\kappa)}\int_{\Gamma_1}a_u(t)w_t(t)\int_0^{\infty}g(\tau)\zeta^t(\tau)d\tau\\
		&&+\frac{2\epsilon}{(1-\kappa)}\int_{\Gamma_1}\left[\,a_u(t)-a_v(t)\,\right]v_t(t)\int_0^{\infty}g(\tau)\zeta^t(\tau)d\tau\\
		&\le&	\epsilon C_R\int_0^{\infty}g(\tau)\|\zeta^t(\tau)\|_{\Gamma_1}\|w_t(t)\|_{\Gamma_1}d\tau+\epsilon C_R\int_0^{\infty}g(\tau)\|\zeta^t(\tau)\|_{\Gamma_1}\|w(t)\|_{\Gamma_1}d\tau\\
		&\le&\frac{\epsilon m_0}{4} \|w_t(t)\|^2_{\Gamma_1}+\frac{\epsilon \kappa}{8}\|\Delta w(t)\|^2+\epsilon C_R\|\zeta^t\|^2_{L^2_g(\mathbb{R}^+;V)}.
	\end{eqnarray*}
	Finally, from \eqref{kernel_g'} of Assumption \ref{A1}, embedding $V\hookrightarrow L^{2}(\Omega)$, and Young's inequality, we have
	\begin{eqnarray*}
		J_5&=&-	\frac{2\epsilon}{(1-\kappa)}\int_0^{\infty}g'(\tau)\int_{\Omega}\zeta^t(\tau)w_t(t)dxd\tau\\
		&\le &	\frac{2\epsilon\alpha_1}{(1-\kappa)}\int_0^{\infty}g(\tau)\|\zeta^t(\tau)\|\|w_t(t)\|d\tau\\
		&\le &	\frac{2\epsilon\alpha_1\lambda_2^{1/2}}{(1-\kappa)}\int_0^{\infty}g(\tau)\|\Delta \zeta^t(\tau)\|\|w_t(t)\|d\tau\\
		&\le&\frac{\epsilon}{2}\|w_t(t)\|^2+\frac{2\epsilon\alpha_1^2\lambda_2}{(1-\kappa)}\|\zeta^t\|^2_{L^2_g(\mathbb{R}^+;V)}.
	\end{eqnarray*}
	Using the estimates of $J_1,\cdots,J_5$ in \eqref{asymp-5} we find the following inequality
	\begin{eqnarray}\label{asymp-6}
		\frac{2\epsilon}{(1-\kappa)}\varphi_2'(t)&\le&-\frac{3\epsilon}{2}\|w_t(t)\|^2+\frac{\epsilon\kappa}{4}\|\Delta w(t)\|^2+\frac{\epsilon m_0}{4} \|w_t(t)\|^2_{\Gamma_1}\nonumber\\
		&&+\epsilon C_R\|\zeta^t\|^2_{L^2_g(\mathbb{R}^+;V)}+\epsilon C_R\|w(t)\|^2_{2(p+1),\Gamma_1}.
	\end{eqnarray}
	\subsubsection*{Completion of the proof: application of the Gronwall's Lemma.}
	Thus, using the estimates \eqref{asymp-2}, \eqref{asymp-4}, and \eqref{asymp-6}, returning to \eqref{asymp-0}, we obtain that
	\begin{eqnarray}
		\label{asymp-7}
		\frac{d}{dt}G(t)&\le& -\frac{\epsilon}{2}\|w_t(t)\|^2-\frac{\epsilon\kappa}{2}\|\Delta w(t)\|^2-\frac{1}{2}\left(\alpha_2-\epsilon C_R\right)\|\zeta^t\|^2_{L^2_g(\mathbb{R}^+;V)}\nonumber\\
		&&-\frac{m_0}{2}(1-\epsilon)\|w_t(t)\|^2_{\Gamma_1}+C_{R,\epsilon}\|w(t)\|^2_{2(p+1),\Gamma_1}.
	\end{eqnarray}
	Again, choosing $\epsilon>0$ small such that
	$$\alpha_2-\epsilon C_R\ge \epsilon\quad \mbox{and}\quad 1-\epsilon \ge 0,$$
	it follows from \eqref{asymp-7} that
	\begin{align*}
		\frac{d}{dt}G(t)+\frac{\epsilon}{2}||S_tz^1_0-S_tz^2_0||^2_{\mathcal{H}}\le C_{R,\epsilon}\|w(t)\|^2_{2(p+1),\Gamma_1}.
	\end{align*}
	From \eqref{equiv}, we have
	\begin{align}
		\label{asymp-8}
		\frac{d}{dt}G(t)+\frac{\epsilon}{3}G(t)\le C_{R,\epsilon}\|w(t)\|^2_{2(p+1),\Gamma_1}.
	\end{align}
	Applying Gronwall's lemma in \eqref{asymp-8}, we obtain
	\begin{align*}
		G(t)\le e^{-\frac{\epsilon}{3}t}G(0)+C_{R,\epsilon}\int_0^te^{-\frac{\epsilon}{3}(t-s)}\|w(s)\|^2_{2(p+1),\Gamma_1}ds.
	\end{align*}
	Therefore, Again using the equivalence \eqref{equiv} we obtain
	\begin{align*}
		||S_tz^1_0-S_tz^2_0||^2_{\mathcal{H}}\le 3e^{-\frac{\epsilon}{3}t}||z^1_0-z^2_0||^2_{\mathcal{H}}+C_{R,\epsilon}\int_0^te^{-\frac{\epsilon}{3}(t-s)}\|u(s)-v(s)\|^2_{2(p+1),\Gamma_1}ds.
	\end{align*}
	Which completes the proof of Proposition \ref{prop-stab-ineq}.
\end{proof}

\subsubsection*{\it \underline{Subcritical case $p<p^*:=\frac{3}{n-4}$}}

\begin{lemma}\label{5.5}
\begin{align*}&
H^{1}(\Omega)\hookrightarrow H^{\frac{1}{2}}(\Gamma)\hookrightarrow L^{\frac{2(n-1)}{n-2}}(\Gamma), \ \ n\geq 3,\\
&
H^{2}(\Omega)\hookrightarrow H^{\frac{3}{2}}(\Gamma)\hookrightarrow\begin{cases} L^{\frac{2(n-1)}{n-4}}(\Gamma), \ \ n\geq 5,\\
L^{p}(\Gamma),\ \ \forall p\geq1,\ \ 1\leq n\leq4.\end{cases}
\end{align*}
\end{lemma}
\begin{proof} Taking into account $H^{\frac{1}{2}}(\Gamma)=\gamma_0(H^1(\Omega)),  H^{\frac{3}{2}}(\Gamma)=\gamma_0(H^2(\Omega))$, where $\gamma_0$ is the trace operator from $\Omega$ to $\Gamma$,  and the Sobolev trace embedding: $H^{1}(\Omega)\hookrightarrow L^{\frac{2(n-1)}{n-2}}(\Gamma), H^{2}(\Omega)\hookrightarrow L^{\frac{2(n-1)}{n-4}}(\Gamma)$,    we have
\begin{align*}&
\|u\|_{{\frac{2(n-1)}{n-2}},\Gamma}=\|\gamma_0u\|_{{\frac{2(n-1)}{n-2}},\Gamma}\\
&\leq C\|\gamma_0u\|_{H^{\frac{1}{2}}(\Gamma)}:=C\inf_{u|_\Gamma=\gamma_0u}\|u\|_{H^{1}(\Omega)}\leq C\|u\|_{H^{1}(\Omega)},\ \ n\geq 3;\\
&\|u\|_{{\frac{2(n-1)}{n-4}},\Gamma}=\|\gamma_0u\|_{{\frac{2(n-1)}{n-4}},\Gamma}\\
&\leq C\|\gamma_0u\|_{H^{\frac{3}{2}}(\Gamma)}:=C\inf_{u|_\Gamma=\gamma_0u}\|u\|_{H^{2}(\Omega)}\leq C\|u\|_{H^{2}(\Omega)},\ \ n\geq 5,\\
&\|u\|_{p,\Gamma}=\|\gamma_0u\|_{p,\Gamma}\\
&\leq C\|\gamma_0u\|_{H^{\frac{3}{2}}(\Gamma)}:=C\inf_{u|_\Gamma=\gamma_0u}\|u\|_{H^{2}(\Omega)}\leq C\|u\|_{H^{2}(\Omega)},\ \ 1\leq n\leq  4.
\end{align*}
\end{proof}

\begin{Corollary}\label{Cor-quasi-stable} {\bf (Quasi-stability)}
	If $p<p^*$, the dynamical system $(\mathcal{H},S_t)$ corresponding to problem \eqref{P}-\eqref{initial-condition0} is quasi-stable.
\end{Corollary}
\begin{proof}
	Note that, on the one hand,	using arguments analogous to those used in the proof of the Theorem \ref{theo-existence}-${\bf II.}$, we have
	\begin{align*}
		||S_tz^1_0-S_tz^2_0||^2_{\mathcal{H}}\le a(t)||S_tz^1_0-S_tz^2_0||^2_{\mathcal{H}},\quad \forall t\in[0,T].
	\end{align*}
	where $a(t):=e^{Ct}$. And, on the other hand, from \eqref{inequality-stab} (Proposition \ref{prop-stab-ineq}), we have
	\begin{align*}
		||S_tz^1_0-S_tz^2_0||^2_{\mathcal{H}}\le b(t)||z^1_0-z^2_0||^2_{\mathcal{H}}+c(t)\sup_{0\le s\le t}\left[\,\mu_V(u(s)-v(s))\,\right]^2,
	\end{align*}
	where $\mu_V(u(s)-v(s)):=\|u(s)-v(s)\|_{2(p+1),\Gamma_1}$,  $$b(t):=3e^{-\frac{\epsilon}{3}t},\quad\mbox{and}\quad c(t):=C_{\varepsilon,B}\int_0^te^{-\frac{\epsilon}{3}(t-s)}ds.$$
	Note that, for $p<p^*$ $\mu_{V}$ is a compact seminorm. Indeed, using interpolation inequality with $H^{3/2}(\Gamma_1)\hookrightarrow L^{2(p+1)}(\Gamma_1)\hookrightarrow L^{2}(\Gamma_1)$ and embedding $V\hookrightarrow H^{3/2}(\Gamma_1)$, we have
	\begin{align}\label{5.35}&
		\|u(s)-v(s)\|_{2(p+1),\Gamma_1}\le C\|\Delta u(s)-\Delta v(s)\|^{\theta}\|u(s)-v(s)\|^{1-\theta}_{\Gamma_1}\nonumber\\
&\le C_{B}\|u(s)-v(s)\|^{1-\theta}_{\Gamma_1},
	\end{align}
with $0<\theta<1$ for $1\leq p<p^*$.
 Thus, from \eqref{5.35} and the compacting embedding $V\hookrightarrow\hookrightarrow L^2(\Gamma_1)$. Which shows that $\mu_V$ is a compact seminorm. Moreover,
	since $a(t)$ and $c(t)$ are locally bounded on $\mathbb{R}^+$, and $b(t)\in L^1(\mathbb{R}^+)$ with
	$\lim_{t\rightarrow +\infty}b(t)=0$, it follows from Theorem \ref{quasi-stable-result} that $(\mathcal{H},S_t)$ is a quasi-stable system.
\end{proof}

\section{Exponential attractor}
We finish this work by proving the existence of an exponential attractor for the problem \eqref{P}-\eqref{initial-condition0} whose dimension is finite in some extended space $\widetilde{\mathcal{H}}\supset\mathcal{H}$. First of all, we recall the definition of exponential attractor that can be found in \cite{chueshov-yellow, hale}.
\begin{definition} A compact set $A_{exp}\subset X$ is said to be {\bf inertial (or a fractal exponential attractor)} of a dynamical system $(X,S_t)$ iff $A$ is a positively invariant set of finite fractal dimension and for every bounded set $D\subset X$ there exist positive constants $t_D$, $C_D$ and $\gamma_D$ such that
	$$d_X\left\{S_tD, A_{exp}\right\}\equiv \sup_{x\in D}\mbox{dist}_X\left(S_tx,A_{exp}\right)\le C_De^{-\gamma_D(t-t_D)},\quad t\ge t_D.$$
\end{definition}
The next Theorem establishes conditions for the existence of an exponential attractor for dissipative and quasi-stable systems.
\begin{theorem}[Theorem. 7.9.9, \cite{chueshov-yellow}] \label{theo_attrac_exponential}
	Let $(X,S_t)$ be given by $(\ref{inforce1})$ and satisfying
	$(\ref{inforce2})$. Assume that $({X},S_t)$ is dissipative and
	quasi-stable on some bounded absorbing set $\mathcal{B}$. In
	addition, suppose that there exists a space
	$\widetilde{X}\supseteq X$ such that mapping $t\mapsto S(t)z$ is
	H\"{o}lder continuous in $\widetilde{X}$ for each
	$z\in\mathcal{B},$ that is, there exist $0<\alpha\leq1$ and
	$C_{\mathcal{B},T}>0$ such that
	\begin{align*}
		\| S_{t_2}z-S_{t_1}z\|_{\widetilde{X}} \, \leq \,
		C_{\mathcal{B},T}|t_2-t_1|^{\alpha}, \quad t_1,t_2\in[0,T], \ z
		\in\mathcal{B}.
	\end{align*}
	Then the dynamical system $({X},S_t)$ possesses a (generalized)
	fractal exponential attractor whose dimension is finite in the
	space $\widetilde{X}$.
\end{theorem}

It is known that for elastic bearings such as springs, the forces $f(u)$ follow the {\bf Hook Law}, more specifically, we have $f(u)\equiv \lambda u$, with $\lambda>0$. In this condition, for our problem under study, the operator $A=\Delta^2$ is defined by the triple $\{D(A),L^2(\Omega),a(u,v)\}$,  where
$$a(u,v)=\int_{\Omega}\Delta u\Delta v+\lambda \int_{\Gamma_1}uvd\Gamma$$
with domain
$$D(A)=\left\{v\in V\cap H^4(\Omega);\ \Delta v=0\quad \mbox{on}\quad \Gamma_1\quad \mbox{and}\quad \frac{\partial(\Delta v)}{\partial \nu}=\lambda v\quad \mbox{on}\quad \Gamma_1\right\}.$$
The Spectral Theorem for self-adjoint operators guarantees the existence of a complete orthonormal system $\{\omega_k\}$ of $L^2(\Omega)$ given by the eigenfunctions of $A$. If $\{\lambda_k\}$ are the eigenvalues of operator $A$, then $\lambda_k\rightarrow +\infty$ as $k\rightarrow +\infty$. Using that $A$ is positive, given $\beta>0$ one has
$$D(A^{\beta})=\left\{u\in L^2(\Omega); \sum_{k=1}^{\infty}\lambda_k^{2\beta}|(u,\omega_k)|^2<\infty\right\}$$
and
$$A^{\beta}u=\sum_{k=1}^{\infty}\lambda_k^{\beta}(u,\omega_k)\omega_k,\quad \forall u\in D(A^\beta).$$
In the domain $D(A^{\beta})$ we assume the topology given by
$$\|u\|_{D(A^{\beta})}=\|A^{\beta}u\|_{L^2(\Omega)},$$
where the operators $A^{\beta}$ are self-adjoints, that is,
$$(A^{\beta}u,v)=(u,A^{\beta}v),\quad \forall u,v\in D(A^{\beta})$$
and, in particular $D(A^{1/2})=V$.

\medskip
Now, let $\mathcal{H}_{s}:=D(A^{\frac{s+1}{2}})\times D(A^{s/2})\times L^2_{g}(\mathbb{R}^+;D(A^{\frac{s+1}{2}}))$. In particular $\mathcal{H}\equiv \mathcal{H}_0$.

\begin{proposition}\label{Prop-exp-attractor}
	The dynamical system $(\mathcal{H},S_t)$ possesses a generalized fractal exponential attractor $\mathfrak{A}_{\exp}$ with finite dimension in the extended space
	$$\mathcal{H}_{-s}=D(A^{\frac{-s+1}{2}})\times D(A^{-s/2})\times L^2_{g}(\mathbb{R}^+;D(A^{\frac{-s+1}{2}})),\quad \mbox{with}\quad s\in (0,1].$$
\end{proposition}
\begin{proof}
	Let $\mathfrak{B}$ be the absorbing set obtained in the Proposition \ref{absorbing-set}. From Proposition \ref{Cor-quasi-stable},  the system $(\mathcal{H},S_t)$ is quasi-stable on $\mathfrak{B}$.
	Next, we consider the mapping $t\in[0,T]\mapsto S_tz\in \mathcal{H}_{-1}=L^2(\Omega)\times D(A^{-1/2})\times L^2_{g}(\mathbb{R}^+;L^2(\Omega))$. Let $z(t) = S_tz_0$ be a solution with initial data $z(0)\in \mathfrak{B}$. From the well-posedness result for the problem \eqref{P}-\eqref{initial-condition0} established in the Theorem \ref{theo-existence}, it is easy to show that $z_t=(u_t,u_{tt},\eta^{t}_t)\in L^2(0,T;\mathcal{H}_{-1})$. Then, we have
	\begin{eqnarray*}
		||S_{t_1}z-S_{t_2}z||_{\mathcal{H}_{-1}}&\le& \int_{t_1}^{t_2}\left\|\frac{d}{dt}z(s)\right\|_{\mathcal{H}_{-1}}ds\le \left(\int_{t_1}^{t_2}\|(u_t(s),u_{tt}(s),\eta^{s}_t)\|_{\mathcal{H}_{-1}}^2ds\right)^{\frac{1}{2}}|t_2-t_1|^{\frac{1}{2}}\\
		&\le& C_{\mathfrak{B}}|t_2-t_1|^{1/2},
	\end{eqnarray*}
	$0\le t_1< t_2\le T.$
	This  shows that for each $z\in \mathfrak{B}$, the map $t\mapsto S_tz$ is H\"older continuous in the extended space $\mathcal{H}_{-1}$ with exponent $\gamma=1/2$. Now, we assume that $0<s<1$. Then, using that $\mathcal{H}_{0}\hookrightarrow \mathcal{H}_{-s}\hookrightarrow \mathcal{H}_{-1}$, applying interpolation theorem in each component of $\mathcal{H}_{-s}$ and using the H\"older continuity in $\mathcal{H}_{-1}$, we obtain
	\begin{eqnarray*}
		||S_{t_1}z-S_{t_2}z||_{\mathcal{H}_{-s}}&\le& C_s||S_{t_1}z-S_{t_2}z||^{1-s}_{\mathcal{H}}||S_{t_1}z-S_{t_2}z||^{s}_{\mathcal{H}_{-1}}\\
		&\le& C_{\mathfrak{B}}|t_2-t_1|^{s/2},\quad \mbox{with}\quad s\in (0,1].
	\end{eqnarray*}
	This shows that $t\mapsto S_tz$ is H\"older continuous in the extended spaces $\mathcal{H}_{-s}$ for $s\in (0,1]$. Therefore, the existence of a generalized exponential attractor in $\mathcal{H}_{-s}$ follows from Theorem \ref{theo_attrac_exponential}.
\end{proof}

\subsubsection*{Declarations}

\paragraph{Ethical Approval} 
 Not applicable to this article.
\paragraph{Funding} 
Linfang Liu has been partially supported by NSF of China (No. 11901448) as well as by Scientific Research  Foundation of Northwest University of China. Vando Narciso has been partially supported by Fundect/CNPq (No. 15/2024), Brasil. Zhijian Yang has been partially supported by the National
 Natural Science Foundation   of China (Grant No.12171438).

\paragraph{Availability of data and materials}
Availability of data and materials does not apply to this article as no datasets were generated or analyzed and no materials were used during the current study.

\paragraph{Conflict of interest} On behalf of all authors, the corresponding author states that there is no conflict of interest.


\begin{thebibliography}{16}
	\bibitem{Cavalcanti-2001} M. M. Cavalcanti, V. N. Domingos Cavalcanti, J. S. Prates Filho, J. A. Soriano, Existence and uniform decay rates for viscoelastic problems with nonlinear boundary damping, Diff. and Integral Equations, 14(1), (2001), 85-116.
	
	\bibitem{Cavalcanti-2002} M. M. Cavalcanti, {Existence and uniform decay for Euler-Bernoulli viscoelastic equation with nonlocal boundary dissipation}, Discr. Cont. Dyn. Syst.,  8(3),(2002), 675-695.
	
	\bibitem{chueshov-white} I. Chueshov, I. Lasiecka, {Long-Time Behavior of Second Order Evolution Equations with Nonlinear Damping}, Mem. Amer. Math. Soc., 195, no. 912, Providence, 2008.
	
	\bibitem{chueshov-yellow} I. Chueshov, I. Lasiecka, {Von Karman Evolution Equations: Well-Posedness and Long-Time Dynamics}, Springer Monographs in Mathematics, Springer, New York, 2010.
	
	\bibitem{chueshov} I. Chueshov, {Dynamics of Quasistable Dissipative Systems}, Springer, 2015.
	
	\bibitem{Dafermos} C. M. Dafermos, {Asymptotic stability in viscoelasticity}, Arch. Ration. Mech. Anal., 37 (1970), 297-308.
	
	\bibitem{Feireisl} E. Feireisl, Nonzero time periodic solutions to an equation of Petrovsky type with nonlinear boundary conditions: slow oscillations of beams on elastic bearings, Annali della Scuola Normale Superiore di Pisa, Classe di Scienze, 20(1), (1993), 133-146.
	
	\bibitem{Feckan} M. Feckan, Free vibrations of beams on bearings with nonlinear elastic responses, J. Diff. Eqs., 154(1), (1999), 55-72.
	
	\bibitem{Grossinho-ma} M. R. Grossinho and T. F. Ma, Symmetric equilibria for a beam with a nonlinear elastic foundation, Portugaliae Mathematica, 51 (1994), 375-393.
	
	\bibitem{Grossinho-tersian} M. R. Grossinho, St. A. Tersian, The dual variational principle and equilibria for a beam resting on a discontinuous nonlinear elastic foundation, Nonlinear Analysis: Theory, Methods $\&$ Applications, 41(3), (2000), 417-431.
	
	\bibitem{hale} J. K. Hale, Asymptotic behavior of dissipative systems, sathematical surveys and monographs 25, American Mathematical Society, Providence, 1988.
	
	\bibitem{Lag} J. E. Lagnese, {Boundary Stabilization of Thin Plates},  SIAM Studies in Applied Mathematics, 1989.
	
	\bibitem{Lagnese-1991} J. E. Lagnese and G. Leugering, Uniform stabilization of a nonlinear beam by nonlinear boundary feedback, J. Diff. Eqs., 91, (1991), 355-388.
	
	\bibitem{Lange-menzala} H. Lange and G. P. Menzala, Rates of decay of a nonlocal beam equation, Diff. and Integral Equations, 10(6), (1997), 1075-1092.
	
	\bibitem{lasiecka-tataru} I. Lasiecka and D. Tataru, {Uniform boundary stabilization of semilinear wave equations with nonlinear boundary damping,} Diff. and Integral Equations, 6 (1993), 507-533.
	
	\bibitem{Irena-Ong} I. Lasiecka and J. Ong, {Global solvability and uniform decays of solution to quasilinear equation with nonlinear boundary dissipation,} Comm. Partial Differential Equations, 24 (1999), 2069-2107.
	
	\bibitem{Ji-lasiecka} G. Ji and I. Lasiecka, Nonlinear boundary feedback stabilization for a semilinear Kirchhoffplate with dissipation acting only via moments-limiting behavior, J. Math. Anal. Appl., 229, (1999), 452-479
	
	\bibitem{Leugering} G. Leugering, Exact boundary controllability of an integrodifferential equation, Appl.  Math. Opt., 15, (1987), 223-250.
	
	
	\bibitem{lions} J. L. Lions, Quelques M\'ethodes de R\'esolution des Probl\'emes aux Limites Non Lin\'eaires, Dunod, Paris, (1969).
	
	\bibitem{lions-magenes} J. L. Lions and E. Magenes,
	Probl\'emes aux Limites Non Homog\'enes et Applications. Vol. 2, Travaux et Recherches Math\'ematiques, No. 18 Dunod, Paris, 1968.
	
	\bibitem{lions-1978} J. L. Lions, On some questions in boundary value problems of mathematical physics, North-holland Mathematics Studies, 30(1978), 284-346.
	
	\bibitem{Matofu-2000} T. F. Ma, Existence results for a model of nonlinear beam with elastic bearings, Appl. Math. Lett., 13 (2000), 11-15.
	
	\bibitem{Matofu-2001} T. F. Ma, Boundary stabilization for a nonlinear beam on elastic bearings, Math. Methods Appl. Sci., 24 (2001), 583-594.
	
	\bibitem{Matofu-narciso-pelicer} T. F. Ma, V. Narciso, M. L. Pelicer,
	Long-time behavior of a model of extensible beams with nonlinear boundary dissipations, J. Math. Anal. Appl.,  396(2), (2012), 694-703.
	
	\bibitem{Rivera} J. Munoz Rivera, E. C. Lapa and R. Barreto, Decay Rates for Viscoelastic plates with Memory, J. of Elasticity, 44(1), (1996), 61-87.
	
	
\end{thebibliography}
\end{document}